\documentclass[11pt]{amsart}

\usepackage[hmargin=0.8in,height=8.6in]{geometry}
\usepackage{amssymb,amsthm, times,bbm,mathrsfs}

\usepackage{delarray,verbatim, soul}\setstcolor{red}

\usepackage{caption}
\usepackage{subcaption}
\usepackage{ifpdf}
\ifpdf
\usepackage[pdftex]{graphicx}

 \else
\usepackage[dvips]{graphicx}
 \fi

\usepackage{bm}
\allowdisplaybreaks
\usepackage{ifpdf}
\usepackage{color}
\definecolor{webgreen}{rgb}{0,.5,0}
\definecolor{webbrown}{rgb}{.8,0,0}
\definecolor{emphcolor}{rgb}{0.5,0.95,0.95}

\DeclareMathAlphabet{\mathpzc}{OT1}{pzc}{m}{it}

\newcommand{\Pro}{\mathbbm{P}}
\renewcommand {\L}{\mathcal{L}}

\newcommand{\bu}{\mathbbmss{b}}
\DeclareMathOperator{\expo}{e}
\newcommand{\eqdef}{\raisebox{0.4pt}{\ensuremath{:}}\hspace*{-1mm}=}

\newcommand{\br}{\mathpzc{b}}

\newcommand{\Cr}{\mathpzc{c}}

\newcommand {\ud}{{\rm d}}

\numberwithin{equation}{section}

\newtheorem{theorem}{Theorem}[section]
\newtheorem{proposition}{Proposition}[section]

\newtheorem{remark}{Remark}[section]

\newtheorem{lemma}{Lemma}[section]

\newtheorem{assump}{Assumption}[section]
\newtheorem*{algo*}{Algorithm}
\newtheorem{algo}{Algorithm}

\numberwithin{remark}{section}
 \numberwithin{proposition}{section}
\numberwithin{corollary}{section}
\newcommand {\R}{\mathbbm{R}}

\newcommand {\E}{\mathbbm{E}}

\newcommand{\diff}{{\rm d}}

\newcommand{\lev}{L\'{e}vy }

\newcommand{\green}{\textcolor[rgb]{0.00,0.50,0.50}}

\newcommand{\der}{\mathrm{d}}

\begin{document}
\title[An optimal multibarrier strategy for a singular stochastic control problem]{
An optimal multibarrier strategy for a singular stochastic control problem with a state-dependent reward}

\author{Mauricio Junca}
\address{Department of Mathematics, Universidad de los Andes, Bogot\'a, Colombia.}
\email{mj.junca20@uniandes.edu.co}

\author{Harold A. Moreno-Franco}
\address{Department of Statistics and Data Analysis, Laboratory of Stochastic Analysis and its Applications, National Research University Higher School of Economics,
Moscow, Russian Federation.}
\email{hmoreno@hse.ru}

\author{Jose-Luis P\'erez}
\address{Department of Probability and Statistics, Centro de Investigacion en Matem\'aticas A.C., Guanajuato, Mexico.}
\email{jluis.garmendia@cimat.mx}

\maketitle

\begin{abstract}
We consider a singular control problem that aims to maximize the expected cumulative rewards, where the instantaneous returns depend on the state of a controlled process. The contributions of this paper are twofold. Firstly, to establish sufficient conditions for determining the optimality of the one-barrier strategy when the uncontrolled process $X$ follows a spectrally negative L\'evy process with a L\'evy measure defined by a completely monotone density. Secondly, to verify the optimality of the $(2n+1)$-barrier strategy when $X$ is a Brownian motion with a drift. Additionally, we provide an algorithm to compute the barrier values in the latter case.
\end{abstract}

\section{Introduction and problem formulation}

Singular stochastic control problems have applications in various fields such as finance, actuarial sciences, and harvesting; see e.g. \cite{A1998,AvPaPi07, Loeffen082,LunguOksendal}. In this applied literature, the main objective is to establish suitable conditions for determining an optimal strategy to maximize expected  rewards until the controlled process falls below zero for the first time. The focus often centres on seeking nearly explicit solutions for these optimization problems.  

An example of such a scenario is seen in optimal dividend problems when the  underlying process $X$ is a spectrally negative  L\'evy process whose L\'evy  measure has a  completely monotone density (see  Assumption \ref{assump_completely_monotone}). Under the condition that the instantaneous rewards (IRs) for executing an admissible strategy are independent of the state of the controlled process, Loeffen \cite{Loeffen082}  showed that the one-barrier strategy is optimal. A similar situation is found in harvesting problems, but in this case the uncontrolled process $X$ is described by a diffusion process without jumps; see e.g. \cite{A1998,LunguOksendal}. 

When the IRs depend on the state of the controlled process through a function $g$ (see \eqref{f1}), known as the {\it instantaneous marginal yield function}, Alvarez studied the problem  in \cite{A2000} for the case where $X$ is a linear diffusion process. Under certain assumptions on the function $g$, the underlying diffusion $X$ and the parameters of the problem, Alvarez  showed that  an optimal control strategy is achieved  with an one-barrier strategy for further details see Remark \ref{remark_alvarez}). Analogous conclusions have been reached in related contexts, such as the optimal bail-out dividend problem for one-dimensional diffusions, as detailed in \cite{F2019}.

Other works on singular control problems that include the dependency on  the state of the controlled process on the IRs  without providing a specific form of the optimal policy are \cite{AR2009,Song}. The difficulty is that incorporating state-dependent IRs  introduces considerably more challenges compared to situations with constant IRs. Unlike the latter, in the former case, the structure of the optimal strategy and the criteria for optimality depend not only on the characteristics of the underlying process, typically encoded in the properties of the scale functions, but also upon the nature of the function $g$ as  discussed in \cite{A2000} when $X$ is a linear diffusion process.

In this study, our first contribution is to expand the scenarios in which the  one-barrier strategy is optimal for the maximization problem mentioned earlier. We specifically consider sufficiently smooth positive functions $g$ that satisfy Assumption \ref{assumpg} (which depends on the parameters of the problem), and spectrally negative L\'evy processes where the  L\'evy  measure has a completely monotone density.

However, in situations where some of the aforementioned conditions are not met, there is no guarantee that one-barrier strategies are optimal. In fact, in the context of the optimal dividend problem, there is a well-known example in \cite{Azcue} where a multibarrier strategy is optimal. When $X$ is a Cramer-Lundberg process,  Gerber showed that an optimal strategy is of the multibarrier type (see \cite{Schmidli}, Section 2.4.2), however, there was no procedure to calculate the actual values of the barriers. Some works have proposed methodologies for identifying optimal barriers, with \cite{AlbrecherFlores} representing a recent endeavor in this regard.

Hence, the second and main contribution of this work lies in presenting an algorithm for determining optimal barriers for general sufficiently smooth positive functions $g$ that satisfy Assumption \ref{assumpg}, and with a Brownian motion with drift serving as the underlying process $X$.

One of the main applications of our model lies in the classical theory of ruin, which examines the path of a stochastic process until the occurrence of ruin defined as the first instance when the process falls below zero. In this context, our model can be viewed as an extension of de Finetti's dividend problem (see \cite{DeFinetti}). The de Finetti problem entails an insurance company making dividend payments to shareholders throughout its operational lifespan until the event of ruin. These dividend payments are intended to maximize the expected net present value of the dividend payments to the shareholders up to the time of ruin. In the spirit of \cite{Sch17}, we generalize this classic problem by incorporating instantaneous state-dependent transaction costs or taxes. After these costs are deducted, shareholders receive a state-dependent instantaneous net proportion of the surplus, given by the function $g$, in the form of dividends. Further results on the classical optimal dividend problem driven by spectrally negative L\'evy process can be found on \cite{AvPaPi07,Loeffen082,Loeffen08}. Additional related results for a class of diffusion process considering policies with transaction costs can be found on \cite{BP2010,BP2012} and with capital injections over a finite horizon in \cite{FS2019}.

Another application of our model is in the context of the harvesting problem (see for instance \cite{LunguOksendal,Song}). In this case, the stochastic process $X$ represents the size or density of a population, with the net price per unit for the population being state-dependent and given by the function $g$. In the context of ecology it makes sense that the price of harvesting is not constant. For instance, harvesting costs tend to be higher for smaller populations since locating specific individuals for harvesting is more challenging. Conversely, in larger populations, harvesting costs tend to be lower as individuals are more readily located for harvesting purposes. The objective of this problem is to identify the optimal harvesting policy that maximizes the total expected discounted income harvested until the population becomes extinct.

\subsection{Problem formulation}

Let  $X=\{X_{t}:t\geq0\}$ denote a spectrally negative L\'evy process, i.e., a L\'evy process with non-monotone trajectories that only has negative jumps, defined on a probability space $(\Omega,\mathcal{F},\Pro)$. For $x\in\R$ we denote by $\mathbb{P}_x$ to the law of the process $X$ starting at $x$ and for simplicity, we write $\mathbb{P}$ instead of $\mathbb{P}_0$. Accordingly, we use $\E_x$ (resp. $\E$) to denote expectation operator associated to the law $\mathbb{P}_x$ (resp. $\mathbb{P}$). Additionally, we denote the natural filtration generated by the process $X$ by $\mathbb{F}\eqdef \{\mathcal{F}_t:t\geq0\}$ satisfying the usual conditions.

A strategy  $\pi\eqdef \{L_t^{\pi}:t\geq0\}$ is a non-decreasing, right-continuous, and $\mathbb{F}$-adapted process. Hence, for each strategy $\pi$, the controlled process $X^{\pi}=:\{X_t^{\pi}:t\geq 0\}$ becomes
\[
X_t^{\pi}\eqdef X_t-L_t^{\pi},\quad \text{for\ $t\geq0$}. 
\]
We write $\tau^{\pi}\eqdef\inf\{t>0:X_t^{\pi}<0\}$ for the time of ruin, we say that a strategy is admissible if the controlled process $X^\pi$ is not allowed to go below zero by the action of the control $L^{\pi}$, that is, if $L_t^{\pi}-L_{t-}^{\pi}\leq X_{t-}^{\pi}$ for $t<\tau^{\pi}$, where $L_{0-}^{\pi}:=0$ and $X_{0-}^{\pi}:=x\geq0$. The set of admissible strategies is denoted by $\mathcal{S}$. The expected reward (ER) for each admissible strategy $\pi\in\mathcal{S}$ is given by
\begin{equation}\label{f1}
V_\pi(x)\eqdef\E_x\left[\int_0^{\tau^{\pi}}\expo^{-qs}g(X^\pi_{s-})\circ\ud L^\pi_s\right],\quad x\geq 0,
\end{equation}
where $q>0$ is the discount rate,  $g:(0,\infty)\mapsto(0,\infty)$ is a twice continuously differentiable function,  and
\begin{align}\label{p1}
\int_0^{\tau^{\pi}}\expo^{-qs} g(X^\pi_{s-})\circ\ud L^\pi_s&\eqdef\int_0^{\tau^{\pi}}\expo^{-qs} g(X^{\pi}_{s})\ud L^{\pi,c}_{s}\\
&\quad+\sum_{0\leq s\leq \tau^\pi}\expo^{-q s}\Delta L^{\pi}_{ s}\int_{0}^{1} g  (X^{\pi}_{ s-}+\Delta X_{s}-\lambda\Delta L^{\pi}_{ s})\ud\lambda,\notag
\end{align}
where $L^{\pi,c}$ denotes the continuous part of $L^{\pi}$. Note that the strategy $\pi$ generates two types of rewards: the first one is related to the continuous part $L^{\pi,c}$, and the other to the jumps of the process $L^\pi$.

We want to maximize the performance criterion  over the set of all admissible strategies $\mathcal{S}$ and find the value function of the problem given by
\begin{align}\label{f2}
V(x)=\sup_{\pi\in\mathcal{S}}V_{\pi}(x),\quad\text{$x\geq0$}.
\end{align}

By the dynamic programming principle,  we identify heuristically that $V$ is associated with  the following Hamilton-Jacobi-Bellman (HJB) equation
\begin{equation}\label{hjb}
\max\{(\L-q)V,g-V'\}=0\ \text{ on}\ [0,\infty),
\end{equation}
where $\L$ is the infinitesimal generator of the process $X$  as in \eqref{generator} below. 

To solve the optimization problem given in \eqref{f2}, we first propose a candidate strategy $\pi^{*}$ for being the optimal strategy (see Eqs. \eqref{lower_barrier} and \eqref{algo4}), then verify, under some assumptions, that its ER  $V_{\pi^{*}}$ satisfies \eqref{hjb}, and finally, using a verification lemma, conclude that $\pi^{*}$ is the optimal strategy for the problem; see Sections \ref{oneb} and \ref{multib}. To this end, we will compute $V_{\pi^{*}}$ and express it in terms of the so-called scale functions; see Section \ref{qsf1}. 

\section{Scale Functions}\label{qsf1}

In this section, we will review $q$-scale functions and provide some properties that will be used throughout this work.  Let us begin with the Laplace exponent $\psi:[0,\infty)\mapsto\R$ of the process $X$, defined by
\begin{align*}
\psi(\theta)\eqdef\log\Big[\E\Big[\expo^{\theta X_1}\Big]\Big]=\mu\theta+\dfrac{\sigma^2}{2}\theta^2-\int_{(0,\infty)}\big(1-\expo ^{-\theta z}-\theta z\mathbf{1}_{\{0<z<1\}}\big)\Pi(\ud z),
\end{align*}
where $\mu\in\R$, $\sigma\geq0$ and the L\'evy measure of $X$, $\Pi$, is a measure defined on $(0,\infty)$ satisfying  $$\int_{(0,\infty)}(1\wedge z^{2})\Pi(\der z)<\infty.$$

For each $q\geq0$, the $q$-scale function of $X$ is a mapping  $W^{(q)}:\R\longrightarrow[0,\infty)$, which is strictly increasing and continuous on $[0,\infty)$ and  such that $W^{(q)}(x)=0$ for $x<0$. It is characterized by its Laplace transform, given by
\begin{equation*}
\int_{0}^{\infty}\expo^{-\theta x}W^{(q)}(x)\der x=\dfrac{1}{\psi(\theta)-q},\quad \text{for $\theta>\Phi(q)$},
\end{equation*}
where $\Phi(q)=\sup\{\theta\geq0: \psi(\theta)=q\}$. Additionally, we define for $x\in\R$
\begin{align*}
Z^{(q)}(x) &\eqdef  1 + q \int_{0}^{x}W^{(q)}(y)\der y.
\end{align*}

From Proposition 3 in \cite{Pistorius2004} we know that $q$-scale functions are harmonic for the discounted processes on $(0,\infty)$, that is,
\begin{equation}\label{harmonic_sf}
(\mathcal{L}-q)W^{(q)}(x)=0\quad\text{and}\quad(\mathcal{L}-q)Z^{(q)}(x)=0, \quad x>0.
\end{equation}

As is usual in singular control problems for L\'evy processes (see for instance \cite{Loeffen082}) we will assume the following:
\begin{assump} \label{assump_completely_monotone}
	The \lev measure $\Pi$ of the process $X$ has a completely monotone density.  That is, $\Pi$ has a density $\nu$ whose $n$-th derivative $ \nu^{(n)}$ exists for all $n \geq 1$ and satisfies
	\begin{align*}
	(-1)^n \nu^{(n)} (x) \geq 0, \quad x > 0.
	\end{align*}
\end{assump}
The previous assumption allows us to obtain a more explicit form of the $q$-scale function as seen in the following result.
\begin{lemma}[\cite{Loeffen08}, Theorem 2 and Corollary 1]\label{compmon}
	For $q>0$ and under Assumption \ref{assump_completely_monotone}, the $q$-scale function $W^{(q)}$ can be written as 
	\begin{align}\label{rep_sf}
	W^{(q)}(x)=\dfrac{\expo^{\Phi(q)x}}{\psi'(\Phi(q))}-\hat{f}(x),
	\end{align}
	where $\hat{f}$ is a non-negative, completely monotone function given by $\hat{f}(x)=\displaystyle\int_{0+}^{\infty}\expo^{-xt}\hat{\mu}(\diff t)$, where $\hat{\mu}$ is a finite measure on $(0,\infty)$. Moreover, $W^{(q)\prime}$ is strictly log-convex  (and hence convex)  and $W^{(q)}$ is $C^{\infty}$ on $(0,\infty)$. 
\end{lemma}
Let us define the first down- and up-crossing times, respectively, by
\begin{align*}
\tau_a^- \eqdef  \inf \left\{ t > 0: X_t < a \right\} \quad \textrm{and} \quad \tau_a^+ \eqdef  \inf \left\{ t > 0: X_t >  a \right\}, \quad a \in \R_+,
\end{align*}
where we follow the convention that $\inf \emptyset = \infty$. By Theorem 8.1 in \cite{kyprianou2014}, for any $a > b$ and $x \leq a$, 
\begin{align}
\begin{split}
\E_x \left[\expo^{-q \tau_a^+} \mathbf{1}_{\left\{ \tau_a^+ < \tau_b^- \right\}}\right] &= \frac {W^{(q)}(x-b)}  {W^{(q)}(a-b)}, \\
\E_x \left[\expo^{-q \tau_b^-} \mathbf{1}_{\left\{ \tau_a^+ > \tau_b^- \right\}}\right] &= Z^{(q)}(x-b) -  Z^{(q)}(a-b) \frac {W^{(q)}(x-b)}  {W^{(q)}(a-b)}.
\end{split}
\label{l1}
\end{align}
For the next result we recall that by Lemma \ref{compmon} the scale function $W^{(q)}$ is $C^{\infty}$ on $(0,\infty)$.
\begin{remark}\label{in1}
Note that since $W^{(q)\prime}$ is strictly log-convex on $(0,\infty)$ it follows that $\dfrac{W^{(q)\prime\prime}}{W^{(q)\prime}}$ is strictly increasing on $(0,\infty)$. Additionally, as in the proof of Theorem 1 in \cite{Loeffen08} we have that $W^{(q)\prime}$ has a unique minimum denoted by $a^*$, hence $W^{(q)\prime\prime}$ is strictly negative on $(0,a^*)$ and strictly positive on $(a^*,\infty)$. The latter implies that
	\begin{equation*}
	\begin{cases}
	\displaystyle\lim_{u\downarrow0}\dfrac{W^{(q)\prime\prime}(u)}{W^{(q)\prime}(u)}\in[0,\infty), &\text{if}\ a^{*}=0,\\
	\displaystyle\lim_{u\downarrow0}\dfrac{W^{(q)\prime\prime}(u)}{W^{(q)\prime}(u)}\in[-\infty,0), &\text{if}\ a^{*}>0.
	\end{cases}
	\end{equation*}
	On the other hand, using \eqref{rep_sf}  
	\begin{equation}\label{eq1.1}
	\lim_{u\uparrow\infty}\dfrac{W^{(q)\prime\prime}(u)}{W^{(q)\prime}(u)}=\lim_{u\uparrow\infty}\frac{\Phi(q)-\psi'(\phi(q))\Phi(q)^{-1}e^{-\Phi(q)u}\int_{0+}^{\infty}t^2e^{-ut}\hat{\mu}(dt)}{1+\psi'(\phi(q))\Phi(q)^{-1}e^{-\Phi(q)u}\int_{0+}^{\infty}te^{-ut}\hat{\mu}(dt)}=\Phi(q).
	\end{equation} 
\end{remark}
In the case of Brownian motion with drift, i.e. $X=x+\mu t +\sigma B_{t}$, where $B=\{B_{t}:t\geq0\}$ is a Brownian motion, we know (see for instance pg. 10 in \cite{KKRivero2013}) that for $x\geq0$
\begin{align}\label{sf}
W^{(q)}(x)=\frac{1}{\sqrt{\mu^2+2q\sigma^2}}\left[ \expo^{\Phi(q) x}-\expo^{-\zeta_1 x}\right]\quad \text{and}\quad Z^{(q)}(x)=\frac{q}{\sqrt{\mu^2+2q\sigma^2}} \left[ \frac{\expo^{\Phi(q) x}}{\Phi(q)}+\frac{\expo^{-\zeta_1 x}}{\zeta_1}\right],
\end{align}
where
\begin{align*}
\zeta_1=\frac{1}{\sigma^2}\left(\sqrt{\mu^2+2q\sigma^2}+\mu\right)\quad\text{and}\quad\Phi(q)=\frac{1}{\sigma^2}\left(\sqrt{\mu^2+2q\sigma^2}-\mu\right).
\end{align*}
Note that, 
\begin{equation}\label{W0}
W^{(q)\prime}(0+)=\frac{2}{\sigma^2}
\end{equation}
and
\begin{equation}\label{eq1}
\lim_{u\downarrow0}\dfrac{W^{(q)\prime\prime}(u)}{W^{(q)\prime}(u)}=-\dfrac{2\mu}{\sigma^{2}}.
\end{equation} 
In this case, several useful properties of the scale functions are provided in Appendix \ref{appendixlemmas}.

\section{Optimality of one-barrier strategies for spectrally negative L\'evy processes}\label{oneb}

In this section, we will assume that $ X$ is a spectrally negative L\'evy process whose L\'evy measure $\Pi$ has a completely monotone density; as in Assumption \ref{assump_completely_monotone}. We aim to find conditions where one-barrier strategies are optimal. Hence, we will begin this section by computing the ER given in \eqref{f1} in terms of scale functions for an arbitrary one-barrier strategy (see \eqref{vf1}). Then, we will propose a candidate for the optimal barrier $\br^*$ (see \eqref{lower_barrier}) and verify that the barrier strategy at level $\br^*$ is indeed optimal for the optimization problem given in \eqref{f2} under suitable assumptions. Proofs of the results of this section are found in Appendix \ref{appendixbarrier}.

\subsection{Computation of the ER}

A barrier strategy $\pi_{\br}=\{L^{\br}_{t}:t\geq0\}$ at level $\br\geq0$, is defined by
\begin{equation*} 
L_{t}^{\br}=\left(\sup_{0\leq s\leq t}\{ X_{s}-\br\}\right)\vee 0, \quad \text{for}\ t\geq0.
\end{equation*}
Observe that $\pi_{\br}$ is indeed an admissible strategy, which is continuous if $x\leq\br$, and has a unique jump of size $x-\br$ at time zero if $x>\br$, where $x$ is the starting value of $X$. We denote by $X^{\br}_{t}=X_{t}-L_{t}^{\br}$. 

Let $V_{\br}$ be the ER for the barrier strategy at level $\br$, i.e,
\begin{align}\label{au1}
V_{\br}(x)\eqdef\E_x\left[\int_0^{\tau_{\br}}\expo^{-qt}g(X^{\br}_{t-})\circ\ud L^{\br}_t\right],\quad\text{$x\geq0$,}
\end{align}
with $\tau_{\br}\eqdef \displaystyle\inf\{t\geq0: X^{\br}_{t}<0\}$. Using the strong Markov property, we provide an expression for \eqref{au1} in terms of scale functions.
\begin{proposition}\label{aux2}
	Let $V_{\br}$ be as in \eqref{au1}. Then 
	\begin{equation}\label{vf1}
	V_{\br}(x)=
	\begin{cases} 
	\displaystyle \frac{g(\br)}{W^{(q)\prime}(\br )}W^{(q)}(x)&\text{if }\ x \leq \br,\\
	H(x;\br) & \text{if }\ x \geq\br,
	\end{cases}
	\end{equation}
	where
	\begin{equation}\label{H1}
	H(x;u)= G(x)-G(u)+g(u)\dfrac{W^{(q)}(u)}{W^{(q)\prime}(u)}\quad\text{and}\quad G(x)=\displaystyle\int_0^xg(y)dy.
	\end{equation}
\end{proposition}

\subsection{Selection of the optimal threshold} In order to maximize $V_\br(x)$ uniformly in $x$, by looking at equation \eqref{vf1}, we consider the mapping $F:[0,\infty)\mapsto\R$ by
\begin{equation}\label{eqF1}
F(u)=
\dfrac{g(u)}{W^{(q)\prime}(u)}.
\end{equation}
Then, we define our choice of optimal threshold $\br^{*}$ by
\begin{align}\label{lower_barrier}
	\br^{*}:=\sup\{u\geq 0:  F(u)\geq F(x),\ \text{ for all $x\geq0$}\}.
\end{align}
Observe that $\br^{*}\leq a^{*}<\infty$ (with $a^{*}$ as in Remark \ref{in1}) if $g$ is non-increasing on $(0, \infty)$. To guarantee that $b^*<\infty$ we will make the following assumption throughout the paper.
\begin{assump}\label{assumpg}
We assume that $g(z)=o(e^{\phi(q)z})$ as $z\rightarrow\infty$.
\end{assump}
Note that \eqref{rep_sf} implies that $g(z)=o(W^{(q)}(z))$ and also that $g(z)=o(W^{(q)\prime}(z))$, hence  $b^*<\infty$. Now, from \eqref{vf1} we obtain
\begin{equation}\label{der_vf_1}
	V_\br'(u)=
	\begin{cases} 
		\displaystyle \frac{g(\br)}{W^{(q)\prime}(\br)}W^{(q)\prime}(u)&\mbox{if }  u < \br,\\
		g(u) & \mbox{if } u > \br,
	\end{cases}\quad\textrm{ and }\quad V_\br''(u)=
	\begin{cases} 
		\displaystyle \frac{g(\br)}{W^{(q)\prime}(\br)}W^{(q)\prime\prime}(u)&\mbox{if } u < \br,\\
		g'(u) & \mbox{if } u > \br.
	\end{cases}
\end{equation}
Hence, for each $\br\geq0$, $V_{\br}$ is increasing and continuously differentiable on $(0,\infty)$ since $W^{(q)\prime}$ and $g$ are positive on $(0,\infty)$. Additionally, $F$ is continuously differentiable and
\begin{align}\label{der1}
F'(u)= 
\dfrac{g(u)}{W^{(q)\prime}(u)}\Bigg(\dfrac{g'(u)}{g(u)}-\dfrac{W^{(q)\prime\prime}(u)}{W^{(q)\prime}(u)}\Bigg),
\end{align}
then,
\begin{equation}\label{der2}
	\dfrac{g'(\br^{*})}{g(\br^{*})}=\dfrac{W^{(q)\prime\prime}(\br^{*})}{W^{(q)\prime}(\br^{*})}\quad \text{if}\ \br^{*}>0.
\end{equation}
Therefore, from \eqref{der_vf_1} and \eqref{der2}, $V_{\br^*}$ is twice continuously differentiable. We summarize these results in the next lemma.

\begin{lemma}\label{smoothness}
For any $\br\geq0$, $V_{\br}\in C^{1}((0,\infty))\cap C^{2}((0,\infty))\setminus\{\br\})$ and $V_{\br}>0$ is increasing on $(0,\infty)$. Additionally,  $V_{\br^{*}}\in C^{2}((0,\infty))$, with $\br^{*}$ defined as in \eqref{lower_barrier}.
\end{lemma}

\subsection{Verification of optimality}

We show now that the optimality of the stochastic control problem given in \eqref{f2} is achieved by the barrier strategy at level $\br^*$ under Condition \eqref{decrF1} below. Let $\mathcal{L}$ denote the infinitesimal generator of the process $X$, acting on the space of sufficiently smooth functions i.e. $C^1((0,\infty))$ (resp. $ C^2((0,\infty))$) if $X$ is of bounded variation (resp. unbounded variation). The infinitesimal generator is given by
\begin{equation}\label{generator}
	\begin{split}
	\mathcal{L} f(x)= \frac{\sigma^2}{2}f''(x)+\mu f'(x) +\int_{(0,\infty)}[f(x - z)-f(x)+f'(x)z\mathbf{1}_{\{0<z<1\}}]\Pi(\mathrm{d}z). 
	\end{split}
\end{equation} 

\begin{lemma}[Verification Lemma]\label{lv1}  
	Suppose that $\hat{\pi}\in \mathcal{S}$ is such that $V_{\hat{\pi}}$ is sufficiently smooth on $(0,\infty)$, right-continuous at $0$, and, for all $x>0$,
	 \begin{equation}\label{hjb1}
	 	\max\{(\L-q)V_{\hat{\pi}}(x),g(x)-V'_{\hat{\pi}}(x)\}\leq 0,
	 \end{equation}
	Then $V_{\hat{\pi}}(x)=V(x)$ for all $x\geq0$ and, hence, $\hat{\pi}$ is an optimal strategy. 
\end{lemma}

The previous lemma is the main tool to prove the main result of this section.
\begin{theorem}\label{Ta1}
Assume that $\br^{*}<\infty$ with $\br^{*}$ as in \eqref{lower_barrier}, and 
\begin{equation}\label{decrF1}
	\dfrac{g'(u)}{g(u)}\leq\dfrac{W^{(q)\prime\prime}(u)}{W^{(q)\prime}(u)}\quad\text{for $u\in(\br^*,\infty)$}.
\end{equation}
Then the barrier strategy $\pi_{\br^*}$ is optimal and $V(x)=V_{\br^{*}}(x)$ for all $x\geq0$.
\end{theorem}

\begin{remark}\label{der_f_re}
\begin{enumerate}
	\item[(i)]  Notice that  Condition \eqref{decrF1} is equivalent to the assumption that $F'\leq0$ on $[\br^{*},\infty)$.
	\item[(ii)] Taking $g\equiv1$, we see that \eqref{decrF1} is equivalent to
	\begin{equation}\label{cond1}
		W^{(q)\prime\prime}\geq0\quad \text{on}\ (\br^{*},\infty), 
	\end{equation}
	since $W^{(q)\prime}>0$ on $(0,\infty)$. From here and Remark \ref{in1} we have that $\br^{*}=a^{*}$. Notice that \eqref{cond1} is equivalent to the criterion used by Loeffen to verify that $a^*$ is the optimal barrier for the case mentioned above; see Theorem 2 in \cite{Loeffen082}.
	
	\item[(iii)]\label{p10} By \eqref{harmonic_sf}, \eqref{vf1}, and \eqref{hjb1}, it is enough to show that 
	\begin{equation}\label{s1}
		(\L-q)V_{\br^{*}} \leq0\quad \text{on}\ (\br^{*},\infty)
	\end{equation}
	to verify the statement given in Theorem \ref{Ta1}, since $(\L-q)V_{\br^{*}} =0$ and  $g -V'_{\br^{*}} \leq0$ hold  on $(0,\br^{*})$ and $(0,\infty)$, respectively; see Appendix \ref{appendixbarrier} for more details.
\end{enumerate}
\end{remark}

\begin{remark}\label{remark_alvarez}
Consider that the underlying process $X$ is given by the solution to the following stochastic differential equation
	\begin{equation*}
		\der X_{t}=\bar{\mu}(X_{t})\der t	+\bar{\sigma}(X_{t})\der W_{t}\quad t\geq0,
	\end{equation*}
	where $\bar{\mu}:\R\mapsto\R$ and $\bar{\sigma}:\R\mapsto(0,\infty)$  are Lipschitz continuous functions. Defining  $\psi$  as a combination of the fundamental solutions of the ordinary second-order differential equation $[\mathcal{A}-q]u(x)=0$, with $\mathcal{A}$ the infinitesimal generator of $X$ (for more details, see \cite{BS1996}, Ch. II), under the assumption that $g\in C((0,\infty))\cup C^{1}((0,\infty)\setminus D)$, with $D\subset(0,\infty)$ a countable set,  is non-increasing, and that the function $f(x)\eqdef g(x)/\psi'(x)$ has a unique interior maximum at some point $\br^{*}\in(0,\infty)$, where $f$ is non-increasing  on $(\br^{*},\infty)$, Alvarez  in \cite{A2000} showed that  an optimal control strategy is achieved  through the $\br^{*}$-barrier strategy. However, if the condition
	\begin{equation*}
		\dfrac{[\bar{\sigma}(x)]^{2}}{2}g'(x)+\bar{\mu}(x)g(x)-qG(x)\leq0\quad \text{for}\ x\in(0,\infty)\setminus D,
		\end{equation*}
	holds, he verified that the optimal policy is to drive the process instantaneously to the origin. In our case, notice that this condition is equivalent to checking that \eqref{s1} holds on $(0,\infty)$. 
\end{remark}

\subsubsection{Examples}
In the next two examples it is straightforward to verify that $g$ meets Assumption \ref{assumpg}.
\begin{enumerate}
\item Consider $g(x)=x^{\alpha}$ with $\alpha>0$ fixed. 
denote 
\begin{equation}\label{h3}
h(x)=\frac{g'(x)}{g(x)}-\dfrac{W^{(q)\prime\prime}(x)}{W^{(q)\prime}(x)}, \quad x>0.
\end{equation}
Since $\dfrac{g'(x)}{g(x)}=\dfrac{\alpha}{x}$ and $-\dfrac{W^{(q)\prime\prime}(x)}{W^{(q)\prime}(x)}$ are non-increasing, we have that $x\mapsto h(x)$ is also non-increasing satisfying
\[
\lim_{u\downarrow0}h(u)=+\infty \quad \text{and,} \quad  \lim_{u\uparrow\infty}h(u)=-\Phi(q).
\]
Therefore, there must exist a unique $\hat{b}>0$ such that $h(\hat{b})=0$ which, by \eqref{der1} together with the fact that $g$ and $W^{(q)}$ are strictly positive functions in $(0,\infty)$, imply that $\hat{b}$ is the unique root of the mapping $x\mapsto F'(x)$,  and therefore $\hat{b}=\br^*$. Additionally, note that Condition \eqref{decrF1} is satisfied because the function $h$ is non-increasing. Hence, by Theorem \ref{Ta1}  the barrier strategy $\pi_{\br^{*}}$ is optimal.

\item Let $\beta>0$ and $g(x)=\expo^{-\beta x}$, then $h$ is  non-increasing and therefore $F$   has a unique maximum at $\br^*>0$ if 
$$\dfrac{W^{(q)\prime\prime}(\br^*)}{W^{(q)\prime}(\br^*)}=\dfrac{g'(\br^*)}{g(\br^*)}=-\beta,$$ 
which is true only if $a^{*}\neq0$ by Remark \ref{in1} and \eqref{eq1.1}. On the other hand, if $a^{*}=0$, it follows that $\dfrac{g'(x)}{g(x)}=-\beta<\dfrac{W^{(q)\prime\prime}(x)}{W^{(q)\prime}(x)}$ for all $x\in(0,\infty)$, which implies that $\br^*=0$. Hence, Condition \eqref{decrF1} is satisfied, and therefore the barrier strategy $\pi_0$ is optimal by an application of Theorem \ref{Ta1}.
\end{enumerate}

\subsection{Necessity of Condition \eqref{decrF1}}\label{nec_example}

It is important to remark that Condition \eqref{decrF1} is sufficient but not necessary for the optimality of the one-barrier strategy. Indeed,  in the next example, where the uncontrolled process $X$ is a standard Brownian motion, i.e., $X_{t}=x+\mu t+\sigma W_{t}$, we will show the optimality of the one-barrier strategy even when Condition \eqref{decrF1} is not satisfied. For that, by Remark \ref{p10}, it is enough to check that \eqref{s1} holds. Let $g:[0,\infty)\rightarrow[0,1)$ be defined as
\begin{equation}\label{g1}
g(x)=\dfrac{0.3x^2}{0.5x^{3}-0.32x+0.2}\quad\text{for}\ x\geq0,
\end{equation}
and let $q=0.2$, $\mu=2.3$ and $\sigma=2$. Notice that this function also satisfies Assumption \ref{assumpg}. We obtain numerically that $F$ attains its global maximum at $\br^{*}=0.8925$. Condition \eqref{decrF1} is not satisfied but $(\mathcal{L}-q)V_{\br^{*}}(x)<0$  for $x\in(\br^{*},\infty)$, see Figure \ref{1-barropt}. Then,  by Lemma \ref{lv1}, the barrier strategy at the level $\br^*$ is optimal. Note that the function $g$ taking the value $0$ does not pose an issue. For the process $X$, $0$ is regular for $(-\infty,0)$, meaning that when $X$ reaches $0$, ruin occurs instantaneously.

\begin{figure}[t!]
	\centering
	\includegraphics[scale=0.5]{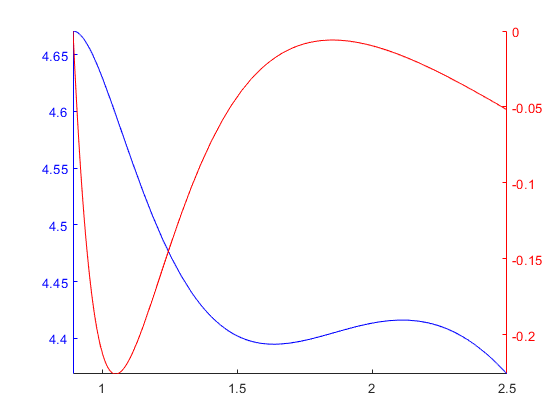}
	\caption{Plots of $F$ (in blue) and $(\mathcal{L}-q)V_{0.8925}$ (in red). The plot starts at the point $\br^*=0.8925$, where $F$ attains its global maximum, but Condition \eqref{decrF1} is not satisfied since $F'$ eventually becomes positive (see Remark \ref{der_f_re}). On the other hand, $(\mathcal{L}-q)V_{0.8925}(x)$ is always negative for $x\geq 0.8925$.}\label{1-barropt}
\end{figure}

Now, if we choose $\mu=2.4$, we get numerically that $F$ attains its global maximum at $\br^{*}=0.9165$. Condition \eqref{decrF1} is not satisfied and there exists $\bar{x}>\br^{*}$ such that $(\mathcal{L}-q)V_{\br^{*}}(\bar{x})>0$, see Figure  \ref{not1optimal}.  In the next section, we will show that for this case the optimal strategy corresponds to a multibarrier strategy.
\begin{figure}[t!]
	\centering
	\includegraphics[scale=0.5]{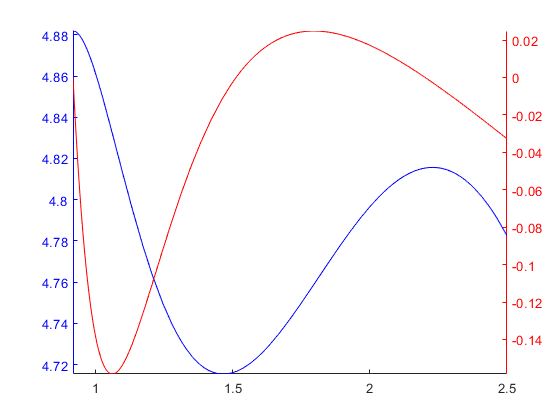}
	\caption{Plots of $F$ (in blue) and $(\mathcal{L}-q)V_{0.9165}$ (in red). The plot starts at the point $\br^*=0.9165$, where $F$ attains its global maximum, but Condition \eqref{decrF1} is not satisfied since $F'$ eventually becomes positive (see Remark \ref{der_f_re}). On the other hand, $(\mathcal{L}-q)V_{0.9165}(\bar{x})$ is positive for some $\bar{x}\geq 0.9165$.}\label{not1optimal}
\end{figure}

\section{Optimality of multibarrier strategies for Brownian motion with drift}\label{multib}

In Subsection \ref{nec_example}, we showed that under a particular choice of parameters,  the one-barrier strategy at level $\br^{*}$ is optimal even when Condition \eqref{decrF1} is not satisfied. However, we can find different choices of parameters under which the HJB inequality given in \eqref{hjb1} is not satisfied for the ER associated with the strategy $\pi_{\br^{*}}$. This suggests that not always the one-barrier strategies are optimal, so it is necessary to find other types of strategies that can achieve optimality.

In this section, we will propose the \textit{$(2n+1)$-barriers strategy} as our candidate for optimality among admissible strategies. Due to the complexity that arises from the jumps of the process $X$ when working with multibarrier strategies, in the remainder of the paper, we will assume that $X$ is a Brownian motion with drift, that is, $\Pi\equiv0$. Proofs of the results of this section are provided in Appendix \ref{appendixmulti}.

We denote the $(2n+1)$-barriers by $\bu=\{\br_{i}\}_{i=1}^{2n+1}$ where $0\leq \br_{1}< \br_{2}<\dots< \br_{2n+1}$, and describe the $\bu$-barrier strategy as follows: If the process lies above $\br_{2n+1}$, push the process to the level $\br_{2n+1}$; if it lies between $(\br_{2k},\br_{2k+1})$, with $k\in\{0,1,\dots,n\}$, do nothing, and finally if it lies between $[\br_{2k+1},\br_{2k+2}]$, with $k\in\{0,1,\dots,n-1\}$, push the process to $\br_{2k+1}$. We denote the $\bu$-barrier strategy by $L^{\bu}$ and the controlled process by $X^{\bu} = X - L^{\bu}$. Formally, if $\bu=\{\br_{i}\}_{i=1}^{2n+1}$, we define the controlled process $X^{\bu}=X-L^{\bu}$ under $\Pro_x$ as follows:  Assume that $x\in[\br_{2k},\br_{2k+2})$ for $k=0,1,\dots,n$, where we take $\br_{0}=0$ and $\br_{2n+2}=\infty$.  
\begin{enumerate}
\item Let $L^{\bu}_t=\left(\sup\limits_{0\leq s\leq t}\{X_s-\br_{2k+1}\}\right)\vee0$ and $X^k=X-L^\bu$, hence $\{X^k_t:t\geq0\}$ is a process reflected at $\br_{2k+1}$ that starts at $x$. Define $\sigma_{k}=\inf\{t\geq0:X^k_t=\br_{2k}\}$, then $X^{\bu}_t=X^k_t$ for $0\leq t\leq\sigma_k$.
\item For $i=k-1,k-2,\ldots,1,0$, define $Z^i_t=X_t-L^{\bu}_{\sigma_{i+1}-}$, so that $Z^i_{\sigma_{i+1}}=\br_{2i+2}$, and let $L^{\bu}_t=L^{\bu}_{\sigma_{i+1}-}+\left(\sup\limits_{\sigma_{i+1}\leq s\leq t}\{Z^i_s-\br_{2i+1}\}\right)\vee0$ and $X^i=Z^i-(L^{\bu}-L^{\bu}_{\sigma_{i+1}-})$. Hence $\{X^i_t:t\geq\sigma_{i+1}\}$ is a process reflected at $\br_{2i+1}$ that starts at $\br_{2i+1}$. Define $\sigma_i=\inf\{t\geq\sigma_{i+1}:X^i_t=\br_{2i}\}$ and $X^{\bu}_t=X^i_t$  for $\sigma_{i+1}\leq t\leq\sigma_i$.
\end{enumerate}

\subsection{Computation of the ER}

Consider the ER associated with the $\bu$-barrier strategy, given by
\begin{align}\label{tn}
V_{\bu}(x)=\E_x\left[\int_0^{\tau_{\bu}}\expo^{-qt}g(X^{\bu}_{t-})\circ\ud L^{\bu}_t\right],\quad\text{for $x\geq0$,}
\end{align}
where $\tau_{\bu}\eqdef \displaystyle\inf\{t\geq0: X^{\bu}_{t}<0\}$. The next result gives the explicit estimations of $V_{\bu}$ in terms of the $q$-scale functions that are associated with the process $X$.
\begin{proposition}\label{c5}
	Let $V_{\bu}$ be as in \eqref{tn}. Then
	\begin{equation}\label{vf5bar}
	V_{\bu}(x)=
	\begin{cases} 
	\displaystyle \frac{g(\br_{1} )}{W^{(q)\prime}(\br_{1} )}W^{(q)}(x)&\mbox{if } x < \br_{1} ,\\
	H(x;\br_{1}) & \mbox{if } x \in [\br_{1} ,\br_{2}],\\
	\qquad\vdots&\qquad\vdots\\
	\phi(x;\bu_{2k+1}) &\mbox{if } x \in (\br_{2k},\br_{2k+1}), \\
	H(x;\bu_{2k+1})&\mbox{if } x \in [\br_{2k+1},\br_{2k+2}],\\
	\qquad\vdots&\qquad\vdots\\
	\phi(x;\bu) &\mbox{if } x \in (\br_{2n},\br_{2n+1}),\\
	H(x;\bu)&\mbox{if } x \in [\br_{2n+1},\infty),
	\end{cases}
	\end{equation}
	where  $H(x;\br_{1})$ is defined in \eqref{H1}, $\bu_{2k+1}\eqdef\{\br_{i}\}_{i=1}^{2k+1}$ for each $k\in\{1,2,\dots,n\}$, and 
	\begin{align}\
	\phi(x;\bu_{2k+1})&:=H(\br_{2k};\bu_{2k-1})Z^{(q)}(x-\br_{2k})\label{phin}\\
	&\quad+W^{(q)}(x-\br_{2k})\bigg(\frac{g(\br_{2k+1})-qH(\br_{2k};\bu_{2k-1})W^{(q)}(\br_{2k+1}-\br_{2k})}{W^{(q)\prime}(\br_{2k+1}-\br_{2k})}\bigg),\notag\\
	H(x;\bu_{2k+1})&:=G(x)-G(\br_{2k+1})+\phi(\br_{2k+1};\bu_{2k+1}).\label{Hn}
	\end{align}
\end{proposition}
Using \eqref{sf} and \eqref{vf5bar}--\eqref{Hn},  we have for  $k=1,\dots,n$, the following conditions:
\begin{align*}
V_{\bu}(\br_1+)&=H(\br_1;\br_1)=g(\br_1)\frac{W^{(q)}(\br_1)}{W^{(q)\prime}(\br_1)}=V_{\bu}(\br_1-),\notag\\
V_{\bu}(\br_{2k}+)&=\phi(\br_{2k};\bu_{2k+1})=H(\br_{2k};\bu_{2k-1})=V_{\bu}(\br_{2k}-),\\
V_{\bu}(\br_{2k+1}+)&=H(\br_{2k+1};\bu_{2k+1})=\phi(\br_{2k+1};\bu_{2k+1})=V_{\bu}(\br_{2k+1}-).
\end{align*}
Additionally, by differentiating \eqref{vf5bar}, we obtain for $k=1,2,\dots,n$:
\begin{align*}
	V_{\bu}'(\br_{1}-)&=\lim_{x\uparrow\br_{1}}g(\br_{1})\frac{W^{(q)\prime}(x)}{W^{(q)\prime}(\br_{1})}=g(\br_{1})=V_{\bu}'(\br_{1}+),\\
	V_{\bu}'(\br_{2k+1}-)&=\lim_{x\uparrow\br_{2k+1}}\bigg[qH(\br_{2k};\bu_{2k-1})W^{(q)}(x-\br_{2k})\\
	&\quad+W^{(q)\prime}(x-\br_{2k})\bigg(\frac{g(\br_{2k+1})-qH(\br_{2k};\bu_{2k-1})W^{(q)}(\br_{2k+1}-\br_{2k})}{W^{(q)\prime}(\br_{2k+1}-\br_{2k})}\bigg)\bigg]\notag\\&=g(\br_{2k+1})=V_{\bu}'(\br_{2k+1}+).
\end{align*}
Hence, $V_{\bu}\in C((0,\infty))\cap C^1((0,\infty)\setminus\{\br_{2k}\}_{k=1}^n)\cap C^{2}((0,\infty)\setminus\bu))$ for any $\bu$. The smooth fit principle requires finding barriers that increase the regularity of this function, which is done next.

\subsection{Selection of an optimal multibarrier strategy}

To select a $\bu^{*}$ as a candidate for being an  optimal barrier,  let us  define the auxiliary function  
\begin{equation}\label{Fn1}
F(v,z;\bu_{2k+1}):=\frac{g(z)-qH(v;\bu_{2k+1})W^{(q)}(z-v)}{W^{(q)\prime}(z-v)}, \quad \text{for}\ (v,z)\in\mathcal{A}_{2k+1},
\end{equation}
where $\bu_{2k+1}=\{\br_{i}\}_{i=1}^{2k+1}$ is  a non-negative increasing sequence and $\mathcal{A}_{2k+1}\eqdef\{(v,z)\in[\br_{2k+1},\infty)\times [\br_{2k+1},\infty):v< z\}$, with $k\in\{0,1,\dots,n\}$ and $n\geq0$.  Since $W^{(q)}\in  C^{\infty}((0,\infty))$, then $F(\cdot,\cdot;\bu_{2k+1})\in  C(\bar{\mathcal{A}}_{2k+1})$, where $\bar{\mathcal{A}}_{2k+1}$ is the closure set of $\mathcal{A}_{2k+1}$, and it is as smooth as $g$ on $\mathcal{A}_{2k+1}$.

The following are two very useful properties of the function $F(\cdot,\cdot,\bu_{2k+1})$ defined in \eqref{Fn1}: 
\begin{enumerate}
	\item[(1)] By \eqref{sf} together with \eqref{W0}, it is easy to verify that for  $k\in\{0,1,\dots,n\}$ and $v\geq\br_{2k+1}$
	\begin{align}\label{eq9p}
	F(v,v+;\bu_{2k+1}):=\lim_{z\downarrow v}F(v,z;\bu_{2k+1})=\frac{\sigma^2}{2}g(v).
	\end{align}
\item[(2)] Now, note that for $z>v$
	\begin{align}\label{dnp}
	&\partial_{z}F(v,z;\bu_{2k+1})=\dfrac{\overline{F}(v,z;\bu_{2k+1})}{W^{(q)\prime}(z-v)},
	\end{align} 
	where
	\begin{equation}\label{dnp2}
	\overline{F}(v,z;\bu_{2k+1})\eqdef g'(z)-qH(v;\bu_{2k+1})W^{(q)\prime}(z-v)-W^{(q)\prime\prime}(z-v)F(v,z;\bu_{2k+1}).
	\end{equation} 
	From \eqref{H1}, \eqref{generator} and \eqref{Hn}, notice that $(\mathcal{L}-q)H(v;\bu_{2k+1})$ can be rewritten as 
	\begin{equation}\label{eqH1}
	(\mathcal{L}-q)H(v;\bu_{2k+1})=\dfrac{\sigma^{2} }{2}g'(v)+\mu g(v)-qH(v;\bu_{2k+1}).
	\end{equation}
	Hence, letting $z\downarrow v$ in \eqref{dnp} and using \eqref{W0} and \eqref{eq1} together with \eqref{dnp2} we obtain
	\begin{align}\label{eq11p}
	\partial_z F(v,v+;\bu_{2k+1})&=\frac{\sigma^2}{2} \overline{F}(v,v+;\bu_{2k+1})\notag\\
	&=\frac{\sigma^2}{2}\left(g'(v)-q\frac{2}{\sigma^2}H(v;\bu_{2k+1})+\frac{4\mu}{\sigma^4}F(v,v+;\bu_{2k+1})\right)\notag\\
	&=(\mathcal{L}-q)H(v;\bu_{2k+1}),
		\end{align}
	where in the last equality we used \eqref{eqH1} and \eqref{eq9p}.
\end{enumerate}

Next, we present a method to select a candidate $\bu^{*}$ for being an optimal multibarrier. Let us first start with the case $n=1$.

\subsection*{Existence of $\bu^{*}_{3}$} 
First, define $\br^{*}_{1}=\br^{*}<\infty$ as in \eqref{lower_barrier}. Now, note that
	\begin{equation}\label{h1}
		(\L-q)H(v_{1};\br^{*}_{1})>0\quad\text{for some}\ v_{1}>\br^{*}_{1},
	\end{equation} 
otherwise, by Remark \ref{der_f_re},  $\bu^{*}=\{\br^{*}_{1}\}$ is an optimal barrier. By definition of $\br^{*}_{1}$, there exists $\br^{(1)}_{1}>\br^{*}_{1}$ such that $F'(v)<0$ for $v\in(\br^{*}_{1},\br^{(1)}_{1})$, that is
\begin{equation}\label{W5}
	g'(v)<W^{(q)\prime\prime}(v)\dfrac{g(v)}{W^{(q)\prime}(v)},
\end{equation}
and
\begin{equation}\label{ex2}
	F(v)>F(z)\quad \text{for}\ v\in(\br^{*}_{1},\br^{(1)}_{1})\ \text{and}\ z>v.
\end{equation}

\begin{proposition}\label{n11}
	\begin{enumerate}
		\item[{\it (i)}] If $F'(\br^{*}_{1})=0$, then
		\begin{equation}\label{HJB11}
			(\mathcal{L}-q)H(\br^{*}_{1};\br^{*}_{1})=0.
		\end{equation} 
		\item[{\it (ii)}] For each $v\in (\br^{*}_{1},\br^{(1)}_{1})$,
		\begin{equation}\label{Hj21}
			(\mathcal{L}-q)H(v;\br^{*}_{1})
			<0.
		\end{equation}
	\end{enumerate}
\end{proposition}

The previous result and \eqref{h1} imply that there exists $\tilde{v}>\br^{(1)}_{1}$ such that
		\begin{equation}\label{ineq1.0.0}
			\partial_{z}F({v},v+;\br^{*}_{1})=(\mathcal{L}-q)H(v;\br^{*}_{1})
			\begin{cases}
				\leq0,&\text{if}\ v\in(\tilde{v}-\varepsilon_1,\tilde{v}),\\
				=0,&\text{if}\ v=\tilde{v},\\
				>0,&\text{if}\ v\in(\tilde{v},\tilde{v}+\varepsilon_1),
			\end{cases}
		\end{equation}	
for some $\varepsilon_1>0$. Hence, the set
\begin{equation}\label{P1}
	\mathcal{P}_{1}\eqdef\{\tilde{v}>\br^{(1)}_{1}: \eqref{ineq1.0.0}\ \text{is true}\ \text{for some}\ \varepsilon_1>0\}
\end{equation}
is not empty. Let $\Cr_{1}\eqdef\min\mathcal{P}_{1}$, which is attained due to the smoothness of $g$.

\begin{lemma}\label{propm1}
	There exists   $\hat{\varepsilon}_{1}>0$, such that
	\begin{equation}\label{par11}
		\partial_{z}F(\Cr_{1},z;\br^{*}_{1})>0,\quad \text{for}\  z\in(\Cr_{1},\Cr_{1}+\hat{\varepsilon}_{1}).
	\end{equation}
\end{lemma}

The lemma implies that the global maximum of $z\mapsto F(\Cr_{1},z;\br^{*}_{1})$, if it exists, is attained at some point $\bar{z}>\Cr_{1}$. Note that under Assumption \ref{assumpg} the global maximum of the functions $z\mapsto F(v,z;\br^{*}_{1})$ exists for all $v>\br^{*}_{1}$. Indeed, this assumption implies that for all $v\geq \br^*_{1}$ we have that $\lim\limits_{z\rightarrow\infty}F(v,z;\br^*_{1})=-\infty$ and therefore the global maximum is attained. Hence, we have that 
\begin{equation}\label{ineq1.10}
		z_{1}(v)\eqdef\sup\{y> v:F(v,y;\br^{*}_{1})\geq F(v,z;\br^{*}_{1}),\ \text{ for all $z> v$}\},\quad \text{for}\ v\in[\br^{*}_{1},\Cr_{1}], 
	\end{equation}
is well defined.  Take the set $\mathcal{D}_{1}$ as follows
	\begin{equation*}
		\mathcal{D}_{1}\eqdef\{v\in[\br_{1}^{*},\Cr_{1}]:v<z_{1}(v)\}.
	\end{equation*}
We see that $\mathcal{D}_{1}$ is the collection of points $v\in[\br^{*}_{1},\Cr_{1}]$ where the map  $z\mapsto F(v,z;\br_{1})$ attains its global maximum at $z=z_{1}(v)>v$. We are now ready to define
\begin{equation*}
		\br_{2}^{*}:=\inf\mathcal{D}_{1}\quad \text{and}\quad \br_{3}^{*}:= z_{1}(\br_{2}^{*}).
\end{equation*}

\begin{proposition}\label{pmax1}
For each $v\in[\br^{*}_{1},\br^{(1)}_{1})$,
	\begin{equation}\label{eq21}
		F(v,v;\br^{*}_{1})>F(v,z;\br^{*}_{1})\quad \text{for}\ z>v.
	\end{equation}
Hence $[\br^{*}_{1},\br^{(1)}_{1})\subset\mathcal{D}_{1}^{c}$, with $\mathcal{D}_{1}^{c}=\{v\in[\br^{*}_{1},\Cr_{1}]:v=z_{1}(v)\}$.
\end{proposition}
 
Proposition \ref{pmax1} shows that $\br_2^*>\br_1^{(1)}>\br_1^*$. Also note that if  $\br^{*}_{2}\in\mathcal{D}_{1}$ then the trajectory $v\mapsto z_{1}(v)$, given by \eqref{ineq1.10}, has a discontinuity at $v=\br^{*}_{2}$. Before proving one of the main theorems of this section, we will illustrate the previous results with a numerical example. Let us consider $g$ as in \eqref{g1} and the surplus process $X$ given by 
\begin{equation*}
X_{t}=x+2.4 t +2B_{t},\quad\text{for}\ t\geq0,
\end{equation*}
with the discount rate $q=0.2$. We already now that $\br^{*}_{1}=0.9165$, see Figure \ref{not1optimal}. To find the next barriers we need to calculate $\Cr_1=1.5113$ and analyze $F(\cdot,\cdot;\br^{*}_{1})$. Figure \ref{F} shows the value of $z_1(v)$ for $v\geq\br_1^*$ and observe that $z_{1}(v)$ has a discontinuity at $\br^{*}_{2}=1.1496$, the minimum $v$ such that $z_1(v)>v$. Also, $\br_3^*=z_1(\br_2^*)=2.1925$.

\begin{figure}[t!]
	\centering
	\begin{subfigure}[b]{0.44\textwidth}
		\centering
		\includegraphics[scale=0.55]{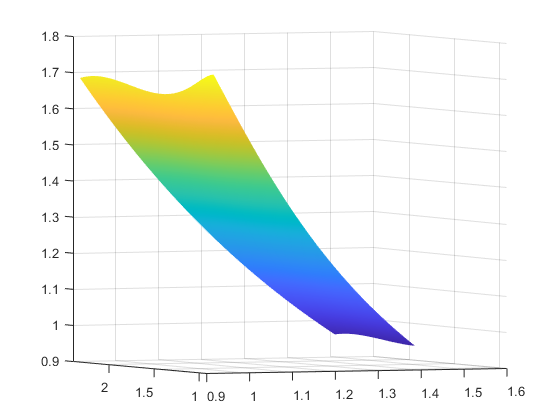}
	\end{subfigure}
	\hfill
	\centering
	\begin{subfigure}[b]{0.55\textwidth}
		\centering
		\includegraphics[scale=0.55]{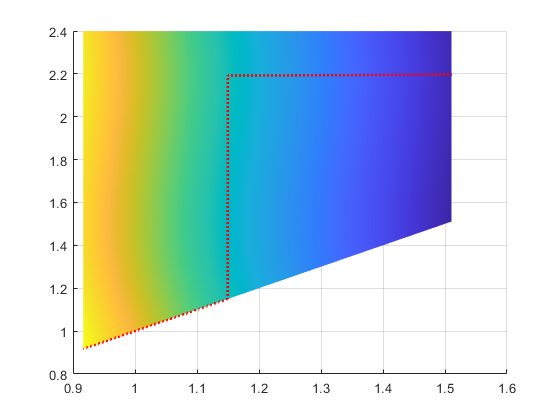}
	\end{subfigure}
	\caption{On the left is the surface of $F(\cdot,\cdot;\br^{*}_{1})$. On the right, the red dotted line is the value of $z_1(v)$ for $v\geq\br^*_1$.}   \label{F} 
\end{figure}

\begin{theorem}\label{prin1}
Assume that there exists $\epsilon_{1}>0$ such that
\begin{equation}\label{cond3}
	\partial_zF(v,v+;\br^*_1)=(\mathcal{L}-q)H(v;\br^*_1)<0,\quad\text{for }v\in(\br^{*}_{2},\br^{*}_{2}+\epsilon_{1}).
\end{equation}

Then 
\begin{equation*}
	\br^{*}_{1}<\br^{*}_{2}<\br^{*}_{3},\quad\br^*_{2}<\Cr_{1}
\end{equation*}
and
\begin{equation}\label{b2}
		F(\br^{*}_{2},\br^{*}_{2};\br^{*}_{1})= F(\br^{*}_{2},\br^{*}_{3};\br^{*}_{1}).
\end{equation}
\end{theorem}

\begin{remark}\label{remcond}
Note that condition \eqref{cond3} is satisfied if the function $v\mapsto(\mathcal{L}-q)H(v;\br^*_1)$ does not remain constant in any open interval, and this is equivalent, by differentiating \eqref{eqH1}, to the condition $(\mathcal{L}-q)g\neq0$ a.e.
\end{remark}

Continuing with the numerical example, Figure \ref{3-barropt} shows in blue the function $F(\br^{*}_{2},\cdot;\br^{*}_{1})$, which satisfies $F(\br^{*}_{2},\br^{*}_{2};\br^{*}_{1})=F(\br^{*}_{2},\br^{*}_{3};\br^{*}_{1})=1.3401$, and in red the function $(\mathcal{L}-q)H(\cdot;\bu^{*}_3)$, which is negative for $x>\br_3^*$. Figure \ref{V} shows the value function $V_{\bu^{*}_{3}}$ and the multibarrier $\bu^{*}_{3}=\{\br^{*}_{1},\br^{*}_{2},\br^{*}_{3}\}$.

\begin{figure}[t!]
	\centering
	\includegraphics[scale=0.5]{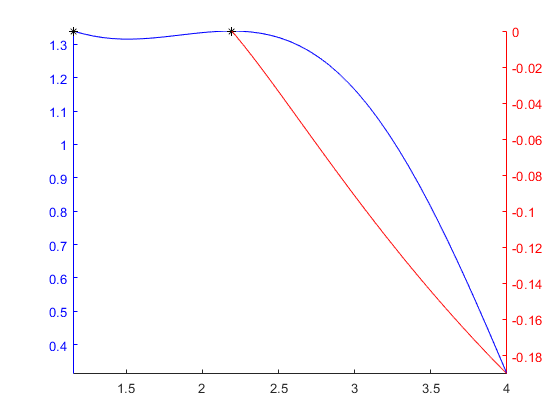}
	\caption{Plots of $F(\br^{*}_{2},\cdot;\br^{*}_{1})$ (in blue) and $(\mathcal{L}-q)H(\cdot;\bu^{*}_3)$ (in red). Black stars show the values of $\br_2 ^*$ and $\br_3 ^*$.}\label{3-barropt}
\end{figure}

\begin{figure}[t!]
	\centering
	\includegraphics[scale=0.5]{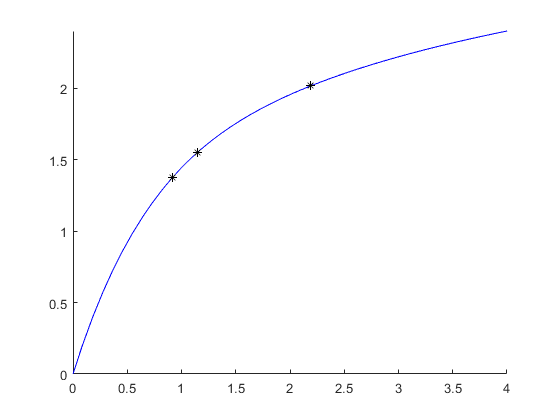}
	\caption{Plot of $V_{\bu^{*}_{3}}$ where the stars show the optimal barriers.}\label{V}
\end{figure}

\subsection*{Existence of $\bu^{*}$}

Now that we have shown the existence of the first three barriers, we can formulate an algorithm that will provide the next two barriers $\br^{*}_{k'}$ with $k'\in\{2k,2k+1\}$ considering that the previous  $\bu^{*}_{2k-1}=\{\br^{*}_{i}\}_{i=1}^{2k-1}$, with $k\in\{2,\dots,n\}$, are obtained.  After that, we will prove the correctness of the algorithm under conditions similar to \eqref{cond3}.
\begin{algo}\label{alg1}
	\begin{enumerate}
		\item\label{algo2} Choose $\br_{1}^{*}$ as in \eqref{lower_barrier}. 
		\item If
		\begin{equation}\label{ineq1}
			(\mathcal{L}-q)H(x;\br^{*}_{1})\leq 0 \quad \text{for}\ x\geq\br^{*}_{1},
		\end{equation}
		let $n=0$ and stop. Otherwise, let $k=1$. 
		\item \label{algo3}Let $\mathcal{P}_{2k-1}$ be the set of points $\tilde{v}>\br^{*}_{2k-1}$ such that for some $\varepsilon_{2k-1}>0$,
		\begin{equation}\label{ineq1.0.0n}
			\partial_{z}F({v},v+;\bu^{*}_{2k-1})=(\mathcal{L}-q)H(v;\bu^{*}_{2k-1})
			\begin{cases}
				\green{\leq}0,&\text{if}\ v\in(\tilde{v}-\varepsilon_{2k-1},\tilde{v}),\\
				=0,&\text{if}\ v=\tilde{v},\\
				>0,&\text{if}\ v\in(\tilde{v},\tilde{v}+\varepsilon_{2k-1}).
			\end{cases}
		\end{equation}	
		Define $\Cr_{2k-1}=\min\mathcal{P}_{2k-1}$   and  
		\begin{equation}\label{Al1n}
			\mathcal{D}_{2k-1}\eqdef\{v\in[\br_{2k-1}^{*},\Cr_{2k-1}]:v<z_{2k-1}(v)\},
		\end{equation}
		where 
		\begin{equation}\label{ineq1.1n}
			z_{2k-1}(v)\eqdef\sup\{y> v:F(v,y;\bu^{*}_{2k-1})\geq F(v,z;\bu^{*}_{2k-1}),\ \text{ for all $z> v$}\}.
		\end{equation}
		Let 
		\begin{equation}\label{algo4}
			\br_{2k}^{*}:=\inf\mathcal{D}_{2k-1}\quad \text{and}\quad \br_{2k+1}^{*}:= z_{2k-1}(\br_{2k}^{*}).
		\end{equation}
		\item If
		\begin{equation}\label{ineq2}
			(\mathcal{L}-q)H(x;\bu^{*}_{2k+1})\leq 0 \quad \text{for}\ x\geq\br^{*}_{2k+1},
		\end{equation}
		choose the barriers $\bu^{*}_{2k+1}$, $n=k$ and stop. Otherwise, let $k=k+1$ and return to (3).
	\end{enumerate}
\end{algo}

To verify the correctness of the algorithm, we assume that, for $k\in\{2,\ldots,n\}$ fixed, the set of barriers  $\bu^{*}_{2k-1}=\{\br^{*}_{i}\}_{i=1}^{2k-1}$ provided by the algorithm,  are well defined, which means that
\begin{align}\label{b1n}
	\begin{split}
		&\br^*_1<\ldots<\br_{2k-3}^{*}<\br_{2k-2}^{*}<\br^{*}_{2k-1},\\
		&F(\br^{*}_{2k-2},\br^{*}_{2k-2};\bu^{*}_{2k-3})= F(\br^{*}_{2k-2},\br^{*}_{2k-1};\bu^{*}_{2k-3}).
	\end{split}
\end{align}
We will show that $\mathcal{P}_{2k-1}$ and $\mathcal{D}_{2k-1}$ are not empty sets, and that $\Cr_{2k-1}$, $\br^{*}_{2k}$ and $\br^{*}_{2k+1}$ are well defined. For that purpose, let us also  suppose that
\begin{equation}\label{h1n}
	(\L-q)H(v_{2k-1};\bu^{*}_{2k-1})>0\quad\text{for some}\ v_{2k-1}>\br^{*}_{2k-1},
\end{equation}
otherwise Algorithm \ref{alg1} would had stop. Now, by definition of $\br^{*}_{2k-1}$ in step (3) of the algorithm, there exists $\br_{2k-1}^{(1)}>\br^{*}_{2k-1}$ such that 
\begin{equation}\label{W5n}
\partial_z F(\br^{*}_{2k-2},v;\bu^{*}_{2k-3})<0\quad\text{for}\ v\in(\br^{*}_{2k-1},\br^{(1)}_{2k-1}),
\end{equation}
with equality at $v=\br^{*}_{2k-1}$. This condition is equivalent to $\overline{F}(\br^{*}_{2k-2},v;\bu^{*}_{2k-3})<0$ with $\overline{F}$ as in \eqref{dnp2}. Moreover, we can choose $\br^{(1)}_{2k-1}$ such that for each $v\in(\br^{*}_{2k-1},\br^{(1)}_{2k-1})$
\begin{equation}\label{ex2n}
F(\br^{*}_{2k-2},z;\bu^{*}_{2k-3})<F(\br^{*}_{2k-2},v;\bu^{*}_{2k-3})\quad\text{for}\ z>v.
\end{equation}
The following result is analogous to Proposition \ref{n11}.

\begin{proposition}\label{n11gen}
	Let $\br^{(1)}_{2k-1}$ be given as before. Then,
	\begin{equation*}
		(\mathcal{L}-q)H(v;\bu^{*}_{2k-1})<0 \quad\text{for}\ v\in(\br^{*}_{2k-1},\br^{(1)}_{2k-1}).
	\end{equation*}
	Hence, $\mathcal{P}_{2k-1}\neq\emptyset$ and $\br^{(1)}_{2k-1}<\Cr_{2k-1}<\infty$. Additionally, $(\mathcal{L}-q)H(\br^{*}_{2k-1};\bu^{*}_{2k-1})=0$.
\end{proposition}

A few straightforward changes in the proof of Lemma \ref{propm1} show that there exists $\hat{\varepsilon}_{2k-1}>0$ small enough, such that for $z\in(\Cr_{2k-1},\Cr_{2k-1}+\hat{\varepsilon}_{2k-1})$
\begin{equation*}
	\partial_{z}F(\Cr_{2k-1},z;\bu^{*}_{2k-1})>0,
\end{equation*}
and therefore $\mathcal{D}_{2k-1}\neq\emptyset$. We also have that $\br^*_{2k-1}<\br^*_{2k}$ from the next result.

\begin{proposition}\label{pmaxn}
For each $v\in[\br^{*}_{2k-1},\br^{(1)}_{2k-1})$,
	\begin{equation}\label{eq21n}
		F(v,v;\bu^{*}_{2k-1})>F(v,z;\bu^{*}_{2k-1})\quad \text{for}\ z>v.
	\end{equation}
	Hence $(\br^{*}_{2k-1},\br^{(1)}_{2k-1})\subset\mathcal{D}_{2k-1}^{c}$, with $\mathcal{D}_{2k-1}^{c}=\{v\in[\br^{*}_{2k-1},\Cr_{2k-1}]:v=z_{2k-1}(v)\}$.
\end{proposition}

The final theorem follows from the above results and the appropriate modifications in the proof of Theorem \ref{prin1}, so we omit its proof.

\begin{theorem}\label{ver1}
Assume that there exists $\epsilon_{2k-1}>0$ such that
\begin{equation}\label{condFin}
	\partial_zF(v,v+;\bu^*_{2k-1})=(\mathcal{L}-q)H(v;\bu^*_{2k-1})<0,\quad\text{for }v\in(\br^{*}_{2k},\br^{*}_{2k}+\epsilon_{2k-1}).
\end{equation}
Then,	$\br^{*}_{2k}$ and $\br_{2k+1}$ given by the Algorithm are well defined and satisfy
	\begin{equation*}
		\br^{*}_{2k-1}<\br^{*}_{2k}<\br^{*}_{2k+1}\quad\text{and}\quad\br^*_{2k}<\Cr_{2k-1}.
	\end{equation*}
	Moreover, 
	\begin{equation}\label{l2.2}
		F(\br^{*}_{2k},\br^{*}_{2k};\bu^{*}_{2k-1})= F(\br^{*}_{2k},\br^{*}_{2k+1};\bu^{*}_{2k-1}).
	\end{equation}
\end{theorem}

Again, by Remark \ref{remcond}, condition \eqref{condFin} is satisfied if $(\mathcal{L}-q)g\neq0$ a.e. Note that if the function $x\mapsto(\mathcal{L}-q)H(x;\bu^{*}_{2k+1})$ is non-increasing for $x\geq\br^{*}_{2k+1}$, or equivalently, if $(\mathcal{L}-q)g(x)\leq0$ for $x\geq\br^{*}_{2k+1}$, then condition \eqref{ineq2} holds. Another sufficient condition is given in the next proposition, which is analogous to Theorem \ref{Ta1}.

\begin{proposition}\label{decFF}
Assume that the function $z\mapsto F(\br^{*}_{2k},z;\bu^{*}_{2k-1})$ is non-increasing for $z\geq\br^{*}_{2k+1}$, then the function $x\mapsto (\L-q)H(x;\bu^{*}_{2k+1})\leq0$ for all $x\geq\br^{*}_{2k+1}$.
\end{proposition}

\subsection{Verification of optimality}

We now verify that the sequence of barriers $\bu^*$ given by Algorithm  \ref{alg1} is indeed optimal. For this, note that the set of barriers $\bu^{*}$ satisfies the following properties:
\begin{itemize}
	\item[(i)] Using \eqref{eq9p} together with \eqref{b2} and \eqref{l2.2} we have for $k=1,\dots,n$
	\begin{align*}
	\frac{\sigma^2}{2}g(\br_{2k}^*)&=F(\br_{2k}^*,\br_{2k}^{*}+;\bu^{*}_{2k-1})=F(\br_{2k}^*,\br_{2k+1}^*;\bu^{*}_{2k-1})\notag\\&=\frac{g(\br_{2k+1}^*)-qH(\br_{2k}^*;\bu_{2k-1})W^{(q)}(\br_{2k+1}^*-\br_{2k}^*)}{W^{(q)\prime}(\br_{2k+1}^*-\br_{2k}^*)}.
	\end{align*}
	Therefore,
	\begin{align*}
	V_{\bu^{*}}'(\br^{*}_{2k}+)&=\frac{2}{\sigma^2}\left(\frac{g(\br_{2k+1}^*)-qH(\br^{*}_{2k};\bu_{2k-1})W^{(q)}(\br_{2k+1}^*-\br^{*}_{2k})}{W^{(q)\prime}(\br_{2k+1}^*-\br^{*}_{2k})}\right)=g(\br^{*}_{2k})=V_{\bu^{*}}'(\br^{*}_{2k}-),
	\end{align*}
	and hence $V_{\bu^{*}}\in C^1((0,\infty))$.
	\item[(ii)] Additionally, by the step \eqref{algo3} 
	\begin{align*}
	\partial_{z}F(\br^{*}_{2k},z;\bu^{*}_{2k-1})\Big|_{z=\br^{*}_{2k+1}}=0.
	\end{align*}
	This implies that 
	\begin{align*}
	g'(\br^{*}_{2k+1})=qH(\br^{*}_{2k},\bu^{*}_{2k-1})W^{(q)\prime}(\br^{*}_{2k+1}-\br^{*}_{2k})+W^{(q)\prime\prime}(\br_{2k+1}-\br^{*}_{2k})F(\br^{*}_{2k},\br^*_{2k+1};\bu^{*}_{2k-1}),
	\end{align*}
	and therefore
	\begin{align*}
	V_{\bu^{*}}''(\br^{*}_{2k+1}+)&=qH(\br^{*}_{2k},\bu^{*}_{2k-1})W^{(q)\prime}(\br^{*}_{2k+1}-\br^{*}_{2k})\notag\\
	&\quad+W^{(q)\prime\prime}(\br^{*}_{2k+1}-\br^{*}_{2k})\bigg(\frac{g(\br^{*}_{2k+1})-qH(\br^{*}_{2k},\bu^{*}_{2k-1})W^{(q)}(\br^{*}_{2k+1}-\br^{*}_{2k})}{W^{(q)\prime}(\br^{*}_{2k+1}-\br^{*}_{2k})}\bigg)\notag\\
	&=g'(\br^{*}_{2k+1})=V_{\bu^{*}}''(\br^{*}_{2k+1}-).
	\end{align*}
	Hence, $V_{\bu^{*}}\in C^2(\R\backslash\{\br_{2k}\}_{k=1}^n)$.
\end{itemize}

Also, the value function $V_{\bu^{*}}$ satisfies that:
\begin{enumerate}
	\item[(i)] As in the proof of Theorem \ref{der_f_re},  $V_{\bu^{*}}$ satisfies the HJB equation \eqref{hjb1} on $(0,\br^{*}_{1})$.
	\item[(ii)] By \eqref{Hn}, $V'_{\bu^{*}}(x)=g(x)$ for $x\in\bigcup_{k=2}^{n}[\br_{2k-1}^*,\br_{2k}^*]\cup[\br^{*}_{2n+1},\infty)$.
	\item[(iii)] By the martingale properties of the scale functions \eqref{harmonic_sf} we also have that $(\mathcal{L}-q)V_{\bu^{*}}(x)=0$ for $x\in\bigcup_{k=1}^{n}(\br_{2k}^*,\br_{2k+1}^*)$.
	\item[(iv)] Since $F(\br_{2k}^*,x;\bu^{*}_{2k-1})\leq F(\br_{2k}^*,\br_{2k+1}^*;\bu^{*}_{2k-1})$ for $x\in(\br_{2k}^*,\br_{2k+1}^*)$ and $k\in\{1,\dots,n\}$, then
	\begin{align*}
	V'_{\bu^{*}}(x)&=qH(\br_{2k}^*;\bu_{2k-1})W^{(q)}(x-\br_{2k}^*)+F(\br_{2k}^{*},\br_{2k+1}^{*};\bu_{2k-1})W^{(q)\prime}(x-\br_{2k}^*)\\
	&\geq qH(\br_{2k}^*;\bu_{2k-1})W^{(q)}(x-\br_{2k}^*)+F(\br_{2k}^{*},x;\bu_{2k-1})W^{(q)\prime}(x-\br_{2k}^*)\\
	&=g(x).
	\end{align*}
	\item[(v)] By construction,  Algorithm \ref{alg1} guarantees that for both $x\in[\br_{2k+1}^*,\br_{2k+2}^*]$, with $k\in\{0,\dots,n-1\}$, and $x\in[\br_{2n+1}^*,\infty)$, we have that
	$$(\mathcal{L}-q)V_{\bu^{*}}(x)=(\mathcal{L}-q)H(x;\bu^{*}_{2k-1})\leq 0.$$
\end{enumerate}
We conclude that $V_{\bu}^{*}$ satisfies the HJB equation \eqref{hjb1}. Now,  by Lemma \ref{lv1}, we obtain the following theorem.

\begin{theorem}
	Let $V$ be the value function given in \eqref{f2}. Then $V_{\bu^{*}}(x)=V(x)$ for all $x\geq 0$ and the $\bu^{*}$-strategy is optimal.
\end{theorem} 

\section*{Acknowledgments}
The authors would like to thank the anonymous reviewers for their comments and suggestions, which helped to improve significantly the quality of this paper.
\section*{Funding}
M. Junca was supported by the Research Fund of the Facultad de Ciencias, Universidad de los Andes INV-2021-128-2307. The authors have no relevant financial or non-financial interests to disclose.

\bibliography{ref2} 

\begin{thebibliography}{10}

\bibitem{AlbrecherFlores}
H.~Albrecher and B.~Garcia~Flores.
\newblock Optimal dividend bands revisited: a gradient-based method and
  evolutionary algorithms.
\newblock {\em Scandinavian Actuarial Journal}, 2023(8):788--810, 2023.

\bibitem{A1998}
L.~H.~R. Alvarez.
\newblock Optimal harvesting under stochastic fluctuations and critical
  depensation.
\newblock {\em Math. Biosci.}, 152(1):63--85, 1998.

\bibitem{A2000}
L.~H.~R. Alvarez.
\newblock Singular stochastic control in the presence of a state-dependent
  yield structure.
\newblock {\em Stochastic Process. Appl.}, 86(2):323--343, 2000.

\bibitem{AR2009}
L.~H.~R. Alvarez and T.~A. Rakkolainen.
\newblock On singular stochastic control and optimal stopping of spectrally
  negative jump diffusions.
\newblock {\em Stochastics. An International Journal of Probability and
  Stochastic Processes}, 81(1):55--78, 2009.

\bibitem{Avram04}
F.~Avram, A.~Kyprianou, and M.~Pistorious.
\newblock Exit problems for spectrally negative {L}\'evy processes and
  application to (canadized) russian options.
\newblock {\em Ann. Appl. Probab.}, 14(1):2115--238, 2004.

\bibitem{AvPaPi07}
F.~Avram, Z.~Palmowski, and M.~Pistorius.
\newblock On the optimal dividend problem for a spectrally negative {L}\'evy
  process.
\newblock {\em Ann. Appl. Probab.}, 17(1):156--180, 2007.

\bibitem{Azcue}
P.~Azcue and N.~Muler.
\newblock Optimal reinsurance and dividend distribution policies in the
  cramÉr-lundberg model.
\newblock {\em Mathematical Finance}, 15(2):261--308, 2005.

\bibitem{BP2010}
L.~Bai and J.~Paulsen.
\newblock Optimal dividend policies with transaction costs for a class of
  diffusion processes.
\newblock {\em SIAM Journal on Control and Optimization}, 48(8):4987--5008,
  2010.

\bibitem{BP2012}
L.~Bai and J.~Paulsen.
\newblock On non-trivial barrier solutions of the dividend problem for a
  diffusion under constant and proportional transaction costs.
\newblock {\em Stochastic Processes and their Applications},
  122(12):4005--4027, 2012.

\bibitem{BS1996}
Andrei~N. Borodin and Paavo Salminen.
\newblock {\em Handbook of {B}rownian motion---facts and formulae}.
\newblock Probability and its Applications. Birkh\"auser Verlag, Basel, second
  edition, 2002.

\bibitem{DeFinetti}
B.~De~Finetti.
\newblock Su un'impostazione alternativa della teoria collettiva del rischio.
\newblock {\em Transactions of the 15th International Congress of Actuaries},
  II:433--443, 1957.

\bibitem{F2019}
G.~Ferrari.
\newblock On a class of singular stochastic control problems for reflected
  diffusions.
\newblock {\em Journal of Mathematical Analysis and Applications},
  473(2):952--979, 2019.

\bibitem{FS2019}
G.~Ferrari and P.~Schuhmann.
\newblock An optimal dividend problem with capital injections over a finite
  horizon.
\newblock {\em SIAM Journal on Control and Optimization}, 57(4):2686--2719,
  2019.

\bibitem{KKRivero2013}
A.~Kuznetsov, A.~E. Kyprianou, and V.~Rivero.
\newblock The theory of scale functions for spectrally negative {L}\'evy
  processes.
\newblock {\em Levy Matters II}, Vol. 2061:97--186, 2013.

\bibitem{KLP}
A.~Kyprianou, R.~Loeffen, and J.~L. P\'erez.
\newblock Optimal control with absolutely continuous strategies for spectrally
  negative l\'evy processes.
\newblock {\em Journal of Applied Probability}, 49(1):150--166, 2012.

\bibitem{kyprianou2014}
A.~E. Kyprianou.
\newblock {\em Fluctuations of {L}\'evy Processes with Applications}.
\newblock Universitext. Springer-Verlag Berlin Heidelberg, 2014.

\bibitem{Loeffen082}
R.~Loeffen.
\newblock On the optimality of the barrier strategy in de {F}inetti's problem
  for spectrally negative {L}\'evy processes.
\newblock {\em Ann. Appl. Probab.}, 18(5):1669--1680, 2008.

\bibitem{Loeffen08}
R.~Loeffen.
\newblock An optimal dividends problem with a terminal value for spectrally
  negative {L}\'evy processes with a completely monotone jump density.
\newblock {\em Journal of Applied Probability}, 46(1):85--98, 2009.

\bibitem{LunguOksendal}
E.M. Lungu and B~{\O}ksendal.
\newblock Optimal harvesting from a population in a stochastic crowded
  environment.
\newblock {\em Mathematical Biosciences}, 145:47--75, 1997.

\bibitem{Pistorius2004}
M.~R. Pistorius.
\newblock On exit and ergodicity of the spectrally one-sided {L}\'evy process
  reflected at its infimum.
\newblock {\em Journal of Theoretical Probability}, 17(1):183--220, 2004.

\bibitem{protter}
P.~Protter.
\newblock {\em Stochastic integration and differential equations}.
\newblock Springer, Berlin, 2nd edition, 2005.

\bibitem{Schmidli}
H.~Schmidli.
\newblock {\em Stochastic Control in Insurance}.
\newblock Springer-Verlang, London, 2008.

\bibitem{Sch17}
H.~Schmidli.
\newblock On capital injections and dividends with tax in a diffusion
  approximation.
\newblock {\em Scandinavian Actuarial Journal}, 2017(9):751--760, 2017.

\bibitem{Song}
Q.~Song, R.~H. Stockbridge, and C.~Zhu.
\newblock On optimal harvesting problems in random environments.
\newblock {\em SIAM Journal on Control and Optimization}, 49(2):859--889, 2011.

\end{thebibliography}
\bibliographystyle{plain}

\appendix

\section{Properties of scale functions of Brownian motion with drift}\label{appendixlemmas}
Considering $X_{t}=\mu t+\sigma B_{t}$, with $\mu\in\R$, $\sigma\geq0$ and $B_{t}$ as a Brownian motion, the $q$-scale functions $W^{(q)}$ and $Z^{(q)}$ are given by the expressions seen in \eqref{sf}. Then, using \eqref{harmonic_sf} we see that $W^{(q)} $ and $Z^{(q)}$ satisfy the following identities
\begin{equation}\label{Wp1}
	\dfrac{\sigma^2}{2}W^{(q)\prime\prime}(v)+\mu W^{(q)\prime}(v)=qW^{(q)}(v)\quad \text{for}\ v\geq0,
\end{equation}
\begin{equation}\label{Zp1}
	\dfrac{\sigma^2}{2}qW^{(q)\prime}(v)+\mu qW^{(q)}(v)=qZ^{(q)}(v)\quad \text{for}\ v\geq0.
\end{equation}

\begin{lemma}
	For $0\leq b< v<z$, the following hold
\begin{align}
&[W^{(q)\prime}(v)]^{2}-W^{(q)}(v)W^{(q)\prime\prime}(v)=\dfrac{4}{\sigma^{4}} \expo^{(\Phi(q)-\xi_{1})v},\label{W2}\\
&W^{(q)\prime\prime}(z-b)W^{(q)\prime}(z-v)-W^{(q)\prime}(z-b)W^{(q)\prime\prime}(z-v)=\dfrac{4q}{\sigma^{4}}\expo^{(\Phi(q)-\xi_{1})(z-v)}W^{(q)}(v-b),\label{W1}\\
&W^{(q)\prime}(z-b)W^{(q)\prime}(z-v)-W^{(q)\prime\prime}(z-v)W^{(q)}(z-b)=\frac{4\expo^{(\Phi(q)-\xi_{1})(z-v)}}{\sigma^{4}}Z^{(q)}(v-b).\label{W7}
\end{align}
\end{lemma}
The identities \eqref{W2}--\eqref{W7} will be used to verify the results given in Propositions \ref{pmax1} and \ref{pmaxn}.

\section{Proofs of one-barrier strategies}\label{appendixbarrier}

\begin{proof}[Proof of Proposition \ref{aux2}]

Observe that if $X_{0}=x>\br$,   $L^{\br}_{0}=x-\br$. Then, by the strong Markov property, we get that
\begin{align}\label{fluc_rp_2}
	V_{\br}(x)&=(x-\br)\int_0^1g(x-\lambda(x-\br))d\lambda+\E_{\br}\left[\int_0^{\tau_{\br}}\expo^{-qt}g(X^{\br}_{t})\circ \ud L^{\br}_t\right]\\
	&=(x-\br)\int_0^1g(x-\lambda(x-\br))d\lambda+V_{\br}(\br)\notag\\
	&=\int_{\br}^xg(u)du+V_{\br}(\br)=G(x)-G(\br)+V_{\br}(\br).\notag
\end{align}
Let us consider $x\in[0,\br]$, then by \eqref{l1} and the strong Markov property,
 \begin{align}\label{fluc_rp_1}
 V_{\br}(x)=\E_x\left[\expo^{-q\tau_{\br}^{+}}\E_{\br}\left[\int_0^{\tau_{\br}}\expo^{-qt}g(X^{\br}_{t})\circ\ud L^{\br}_t\right];\tau_{\br}^+<\tau_0^-\right]=V_{\br}(\br)\frac{W^{(q)}(x)}{W^{(q)}(\br)}.
 \end{align}
We now compute $V_{\br}(\br)$. Let us denote $S_t=\displaystyle\sup_{0\leq s\leq t}(\widetilde{X}_s\vee 0)$ and $Y_t=S_t-\widetilde{X}_{t}$ for $t\geq0$, where the process $\widetilde{X}$ has the same law than $X-\br$. We also consider the first entrance of the process $Y$ to the set $(a,\infty)$ given by $\sigma_a\eqdef \inf\{t\geq 0: Y_t>a\}$, $a\in\R$. By spatial homogeneity of the L\'evy process $X$ and following Proposition 1 in \cite{AvPaPi07} we have that the ensemble $\{X^{\br}_{t},L_{t}^{\br}, t\leq \tau_{\br};X^{\br}_0=x\}$ has the same law as $\{\br-Y_t,S_{t}, t\leq \sigma_{\br};Y_0=\br-x\}$. Therefore
 \begin{align}\label{fluc_1}
V_{\br}(\br)&=\E_{\br}\left[\int_0^{\tau_{\br}}\expo^{-qt}g(X^{\br}_{t})\circ\ud L^{\br}_t\right]=\E_{0}\left[\int_0^{\sigma_{\br}}\expo^{-qt}g(\br-Y_t)\circ\ud S_t\right].
 \end{align}
By the absence of positive jumps of $X$, it is known that the supremum $S$ is a local time for $0$ for the process $Y$. Now, let us denote by $\mathcal{E}$ the set of excursions away from zero of finite length
\begin{equation*}
\mathcal{E}\eqdef \{\epsilon\in \mathcal{C}: \exists \ \zeta=\zeta(\epsilon)>0 \ \text{such that} \ \epsilon_{\zeta}=0 \ \text{and} \ \epsilon_{x}>0 \ \text{for} \ 0<x<\zeta\},
\end{equation*}
where $\mathcal{C}\eqdef \mathcal{C}([0,\infty))$ is the set of all c\`adl\`ag functions on $[0,\infty)$. Additionally, we denote by $\mathcal{E}^{(\infty)}$ the set of excursions with infinite length. We will work with the excursion process $\epsilon\eqdef \{\epsilon_t: t\geq0\}$ of $Y$, with values on $\mathcal{E}\cup\mathcal{E}^{(\infty)}$, given by
\[
\epsilon_t=\{Y_s:S^{-1}_{t-}\leq s<S^{-1}_t\}, \quad \text{if $S^{-1}_{t-}<S^{-1}_t$,}
\]
where $S^{-1}$ denotes the right-inverse of the local time of $Y$ at $0$. Hence, noting that the process $S$ is continuous and non-decreasing, and following the proof of Theorem 1 in \cite{Avram04} we have \begin{align*}
\E_{0}\left[\int_0^{\sigma_{\br}}\expo^{-qt}g(\br-Y_t)\circ\ud S_t\right]&=\E_{0}\left[\int_0^{\sigma_{\br}}\expo^{-qt}g(\br-Y_t)\ud S_t\right]\notag\\&=\E_{0}\left[\int_0^{\infty}\expo^{-qS^{-1}_t}g(\br-\epsilon_{S^{-1}_t})1_{\{\bar{\epsilon}_{ S^{-1}_t}\leq\br, t<S_{\infty} \}}\ud t\right],
\end{align*}
where $\bar{\epsilon}_t=\sup_{s\leq t}\epsilon_s$ for $t>0$ . By the lack of negative jumps of the process $\epsilon$,  we have that $\epsilon_{S^{-1}_t}=0$, hence using \eqref{fluc_1}
\begin{align}\label{fluc_rp_3}
V_{\br}(\br)&=\E_{0}\left[\int_0^{\infty}\expo^{-qS^{-1}_t}g(\br-\epsilon_{S^{-1}_t})1_{\{\bar{\epsilon}_{ S^{-1}_t}\leq\br, t<S_{\infty} \}}\ud t\right]\notag\\
&=g(\br)\E_{0}\left[\int_0^{\infty}\expo^{-qS^{-1}_t}1_{\{\bar{\epsilon}_{ S^{-1}_t}\leq\br, t<S_{\infty} \}}\ud t\right]\notag\\
&=g(\br)\E_{0}\left[\int_0^{\tau_{\br}}\expo^{-qt}\ud S_t\right]\notag\\
&=g(\br)\E_{\br}\left[\int_0^{\tau_{\br}}\expo^{-qt}\ud L^{\br}_t\right]=g(\br)\frac{W^{(q)}(\br)}{W^{(q)\prime}(\br)},
\end{align}
where the last equality follows from the proof of Theorem 1 in \cite{Avram04}. Finally using \eqref{fluc_rp_3} in \eqref{fluc_rp_2} and  \eqref{fluc_rp_1} gives the result.
\end{proof}

\begin{proof}[Proof of Lemma \ref{lv1}]
By the definition of $V$, it follows that $V_{\hat{\pi}}(x)\leq V(x)$ for all $x\geq0$. Now, let us write $w\eqdef V_{\hat{\pi}}$ and we will show that $w(x)\geq V(x)$ for all $x\geq0$. First, assume that $x>0$ and define, for any $\pi\in\mathcal{S}$, $\tau_0^{\pi}\eqdef \inf\{t>0:X_t^{\pi}\leq 0\}$ and denote by $\mathcal{S}_0$ the following set of admissible strategies:
	 \[
	 \mathcal{S}_0\eqdef \left\{\pi\in\mathcal{S}:V_{\pi}(x)=\E_x\left[\int_0^{\tau^{\pi}_0}\expo^{-qs}g(X^\pi_{s-})\circ\ud L^\pi_s\right]\ \text{for all $x>0$}\right\}.
	 \] 
Then, by the proof of Lemma 4 in \cite{KLP} it is enough to show that $w(x)\geq V_\pi(x)$ for any $\pi\in\mathcal{S}_0$. Hence, fix $\pi\in\mathcal{S}_0$ and let $T_n\eqdef \inf\{t>0: X_t^{\pi}>n\ \text{or} \ X_t^{\pi}<1/n\}$.
	  Noting that $X^{\pi}$ is a semi-martingale and that $w$ is sufficiently smooth on $(0,\infty)$, we can use the change of variable/It\^o's formula on $\expo^{-q(t\wedge T_n)}w(X_{t\wedge T_n}^{\pi})$, with $t\geq0$  (see \cite[Thm. 33, pp. 81]{protter}), to obtain
	\begin{align}\label{veri1}
		w(x)&= \expo^{-q(t\wedge T_n)}w({X}^\pi_{t\wedge T_n})- \int_{0}^{t\wedge T_n}\expo^{-qs} (\mathcal{L}-q)w({X}^\pi_{s-}) \mathrm{d}s\notag\\ 
		&\quad+\int_{[0, t\wedge T_n]}\expo^{-qs}w'({X}^\pi_{s-}) \mathrm{d}L_s^{\pi,c}-M_{t\wedge T_n}+J_{t\wedge T_n},
	\end{align}
	where $M\eqdef \{M_t:t\geq0\}$ is a local martingale with $M_0=0$ and $J\eqdef \{J_t:t>0\}$ is given by
	\begin{align*}
		J_t&=\sum_{0\leq s\leq t}\expo^{-qs}\left[w({X}^\pi_{s-}+\Delta X_{s})-w({X}^\pi_{s-}+\Delta X_{s}-\Delta L^{\pi}_{s})\right]1_{\{\Delta L^{\pi}_s\not=0\}}\\
		&=\sum_{0\leq s\leq t}\expo^{-qs}\Delta L^{\pi}_{s}\int_{0}^{1}w'(X^{\pi}_{s-}+\Delta X_{s}-\lambda\Delta L^{\pi}_{s})\der \lambda,\quad t>0.
	\end{align*}
	By assumption, we know that $(\mathcal{L}-q)w\leq0$, $w'\geq g$ and $w\geq0$ on $(0,\infty)$. Hence, taking the expectations in \eqref{veri1} and using \eqref{p1}, yields that
	\begin{align*}
		w(x)\geq \E_x\left[\int_0^{t\wedge T_n}\expo^{-qs}g(X^\pi_{s-})\circ\ud L^\pi_s\right].
	\end{align*}
Letting $t,n\uparrow\infty$ in the previous inequality, using monotone convergence, the fact that $T_n\uparrow \tau_0^{\pi}$ $\mathbb{P}_x$-a. s. and that $\pi\in\mathcal{S}_0$, we have that
		\begin{align*}
		w(x)\geq \E_x\left[\int_0^{\tau^{\pi}_{0}}\expo^{-qs}g(X^\pi_{s-})\circ\ud L^\pi_s\right].
	\end{align*}
	Therefore $w(x)\geq V(x)$ for all $x>0$. Finally, using the fact that $V$ is non-decreasing together with the right continuity of $w$ at zero, we obtain that $V(0)\leq \lim_{x\downarrow0}V(x)\leq\lim_{x\downarrow0}w(x)=w(0)$.
\end{proof}

\begin{proof}[Proof of Theorem \ref{Ta1}]
We have by construction that $\pi_{\br^*}\in\mathcal{S}$, therefore by Lemma \ref{lv1} we only need to show that $V_{\br^*}$ satisfies \eqref{hjb1} for $x>0$. By \eqref{harmonic_sf}
\begin{align*}
(\mathcal{L}-q)V_{\br^{*}}(x)=(\mathcal{L}-q)\left(\frac{g(\br^*)}{W^{(q)}(\br^*)}W^{(q)}(x)\right)=0,\quad\text{$0\leq x\leq \br^*$}.
\end{align*}
On the other hand, \eqref{lower_barrier} gives $\dfrac{g(x)}{W^{(q)\prime}(x)}\leq \dfrac{g(\br^*)}{W^{(q)\prime}(\br^*)}$ for $0\leq x\leq \br^*$, this implies 
\begin{align*}
g(x)-V'_{\br^{*}}(x)=g(x)-g(\br^*)\frac{W^{(q)\prime}(x)}{W^{(q)\prime}(\br^*)}\leq 0,\quad\text{$0\leq x\leq \br^*$}.
\end{align*}
Hence, $V_{\br^{*}}$ satisfies \eqref{hjb1} on $[0,\br^*]$. Meanwhile, using \eqref{vf1} we note that
\begin{equation}\label{eq6}
V'(x)=g(x)\quad  \text{for}\ x \geq\br^*.
\end{equation}
For any $x>\br^*$, let us consider the ER associated with the barrier strategy at the level $x$, $V_x$, which is given in \eqref{vf1}. We recall that, by Lemma \ref{smoothness}, $V_{x}\in C^{1}((0,\infty))\cap C^{2}((0,\infty)\setminus\{x\})$ and $V_{\br^{*}}\in C^{2}((0,\infty))$. We aim to prove that
\begin{equation}\label{claim_gen}
(\mathcal{L}-q)(V_{\br^{*}}-V_x)(x-)\leq 0,\quad x>\br^*.
\end{equation}
To this end, we check the following:
\begin{itemize}
\item[(i)] Condition \eqref{decrF1} implies that $V''_x(x-)=g(x)\dfrac{W^{(q)\prime\prime}(x)}{W^{(q)\prime}(x)}\geq g'(x)=V''_{\br^{*}}(x)$.
\item[(ii)] By \eqref{lower_barrier} we obtain that 
\begin{align*}
	V_x'(u)=\frac{g(x)}{W^{(q)\prime}(x)}W^{(q)\prime}(u)\leq \frac{g(\br^*)}{W^{(q)\prime}(\br^*)}W^{(q)\prime}(u)=V'_{\br^{*}}(u),\quad \text{for $u\in[0,\br^*]$}.
\end{align*}
On the other hand, by Remark \ref{der_f_re} we have  that $\dfrac{g(x)}{W^{(q)\prime}(x)}\leq \dfrac{g(u)}{W^{(q)\prime}(u)}$  for $u\in[\br^*,x]$. Hence 
\begin{align*}
V_x'(u)=\frac{g(x)}{W^{(q)\prime}(x)}W^{(q)\prime}(u)\leq g(u)=V'_{\br^{*}}(u)\quad \text{for}\  u\in[\br^*,x].
\end{align*}
The previous arguments imply that $(V'_{\br^{*}}-V'_{x})(u)\geq0$ for $u\in[0,x]$.
\item[(iii)] By \eqref{lower_barrier} we obtain that
\[
V_x(\br^*)=\frac{g(x)}{W^{(q)\prime}(x)}W^{(q)}(\br^*)\leq \frac{g(\br^*)}{W^{(q)\prime}(\br^*)}W^{(q)}(\br^*)=V_{\br^{*}}(\br^*).
\]
The above inequality together with (ii) implies that $(V_{\br^{*}}-V_x)(x)\geq0$.
\item[(iv)] By \eqref{der_vf_1} we have that $V_{x}'(x)=g(x)=V_{\br^{*}}'(x)$.
\item[(v)] Using (ii) we obtain that $(V_{\br^{*}}-V_{x})(x-z)\leq (V_{\br^{*}}-V_{x})(x)$ for $z\in(0,x)$. Additionally, by (iii) we have $(V_{\br^{*}}-V_{x})(x-z)=0\leq (V_{\br^{*}}-V_{x})(x)$ for $z>x$.
\end{itemize}
Therefore using (i)-(v) and the fact that
\begin{multline*}
	(\mathcal{L}-q)(V_{\br^{*}}-V_x)(x-)=\gamma(V_{\br^{*}}'-V_x')(x)+\frac{\sigma^2}{2}(V_{\br^{*}}''(x)-V''_x(x-))-q(V_{\br^{*}}-V_x)(x)\\
	+\int_{(0,\infty)}[(V_{\br^{*}}-V_x)(x-z)-(V_{\br^{*}}-V_x)(x)-(V_{\br^{*}}'-V_x')(x)z1_{\{0<z<1\}}]\Pi(dz),
\end{multline*}
we obtain \eqref{claim_gen}. So, proceeding like in the proof of Theorem 2 in \cite{Loeffen082} we obtain that
\begin{align}\label{des_gen_sig=0}
(\mathcal{L}-q)V_{\br^{*}}(x)\leq0,\quad\text{for $x\geq \br^*$.}
\end{align}
Hence, by \eqref{eq6} and \eqref{des_gen_sig=0}, $V_{\br^{*}}$ satisfies \eqref{hjb1} on $[\br^*,\infty)$. 
\end{proof}

\section{Proofs of multibarrier strategies}\label{appendixmulti}

\begin{proof}[Proof of Proposition \ref{c5}]
For $x\in[0,\br_{1})$, we obtain, by the strong Markov property, that
\begin{align*}
		V_{\bu}(x)=\E_x\left[\expo^{-q\tau_{\br_{1}}^+}\E_{\br_{1}}\left[\int_{0}^{\tau_{\bu}}\expo^{-qt}g(X^{\bu}_{t-})\circ\ud L^{\bu}_t\right];\tau_{\br_{1}}^+<\tau_0^-\right]=V_{\bu}(\br_{1})\frac{W^{(q)}(x)}{W^{(q)}(\br_{1})},
\end{align*}
and proceeding like in the proof of Proposition \ref{aux2} we have
\begin{align*}
			V_{\bu}(\br_{1})
			=g(\br_{1})\frac{W^{(q)}(\br_{1})}{W^{(q)\prime}(\br_{1})}.
\end{align*}
On the other hand, for $x\in[\br_{1},\br_{2}]$
		\begin{align*}
		V_{\bu}(x)&=(x-\br_{1})\int_0^1g(x-\lambda(x-\br_{1}))d\lambda+\E_{\br_{1}}\left[\int_0^{\tau_{\bu}}\expo^{-qt}g(X^{\bu}_{t-})\circ\ud L^{\bu}_t\right]\\
		&=(x-\br_{1})\int_0^1g(x-\lambda(x-\br_{1}))d\lambda+V_{\bu}(\br_{1})\notag\\
		&=\int_{\br_{1}}^xg(u)du+V_{\bu}(\br_{1})=G(x)-G(\br_{1})+V_{\bu}(\br_{1})=H(x;\br_1).\notag
		\end{align*}
Now, let $x\in(\br_{2},\br_{3})$. By the Strong Markov property and \eqref{l1}, we have that 
		\begin{align}\label{bar2}
		V_{\bu}(x)&=\E_{x}\left[\int_0^{\tau_{\bu}}\expo^{-qt}g(X^{\bu}_{t-})\circ\ud L^{\bu}_t;\tau_{\br_{2}}^-<\tau_{\br_{3}}^+\right]+\E_{x}\left[\int_0^{\tau_{\bu}}\expo^{-qt}g(X^{\bu}_{t-})\circ\ud L^{\bu}_t;\tau_{\br_{3}}^+<\tau_{\br_{2}}^-\right]\\
		&=V_{\bu}(\br_{2})\E_{x}\left[\expo^{-q\tau_{\br_{2}}^-};\tau_{\br_{2}}^-<\tau_{\br_{3}}^+\right]+V_{\bu}(\br_{3})\E_{x}\left[\expo^{-\tau_{\br_{3}}^+};\tau_{\br_{3}}^+<\tau_{\br_{2}}^-\right]\notag\\
		&=H(\br_{2};\br_{1})\left(Z^{(q)}(x-\br_{2})-\frac{Z^{(q)}(\br_{3}-\br_{2})}{W^{(q)}(\br_{3}-\br_{2})}W^{(q)}(x-\br_{2})\right)+V_{\bu}(\br_{3})\frac{W^{(q)}(x-\br_{2})}{W^{(q)}(\br_{3}-\br_{2})}.\notag
		\end{align}
Note that
			\begin{align*}
			V_{\bu}(\br_{3})=\E_{\br_{3}}\left[\int_0^{\tau_{\br_{2}}^{\br_{3}}}\expo^{-qt}g(X^{\br_{3}}_{t})\circ\ud L^{\br_3}_t\right]+V_{\bu}(\br_{2})\E_{\br_{3}}\left[\expo^{-q\tau_{\br_{2}}^{\br_{3}}};\tau_{\br_{2}}^{\br_{3}}<\infty\right],
			\end{align*}
where $X^{\br_{3}}_{t}=X_{t}-L_{t}^{\br_{3}}$,   
			\begin{equation*}
			L_{t}^{\br_{3}}=\bigg(\sup_{0\leq s<t}\{X_{s}-\br_{3}\}\bigg)\vee 0, \quad t\geq0,
			\end{equation*}
and $\tau_{\br_2}^{\br_3}\eqdef \inf\{t\geq0: X^{\br_3}_{t}<\br_2\}$. Hence, by the spatial homogeneity of Brownian motion, we have that
			\begin{align}
			V_{\bu}(\br_{3})=\E_{\br_{3}-\br_{2}}\left[\int_0^{\tau_{0}^{\br_{3}-\br_2}}\expo^{-qt}g(X^{\br_{3}-\br_{2}}_{t}+\br_{2})\circ\ud L^{\br_3-\br_{2}}_t\right]+V_{\bu}(\br_{2})\E_{\br_{3}-\br_{2}}\left[\expo^{-q\tau_{0}^{\br_{3}-\br_2}};\tau_{0}^{\br_{3}-\br_2}<\infty\right].
			\end{align}
Again, proceeding as in the proof of Proposition \ref{aux2}
			\begin{align}
			\E_{\br_{3}-\br_{2}}\left[\int_0^{\tau_{0}^{\br_{3}-\br_2}}\expo^{-qt}g(X^{\br_{3}-\br_{2}}_{t}+\br_{2})\circ\ud L^{\br_3-\br_{2}}_t\right]=g(\br_{3})\frac{W^{(q)}(\br_{3}-\br_{2})}{W^{(q)\prime}(\br_{3}-\br_{2})}.
			\end{align}
On the other hand, by identity (3.10) in \cite{AvPaPi07}, we obtain
			\begin{align}\label{aux_4}
			\E_{\br_{3}-\br_{2}}\left[\expo^{-q\tau_{0}^{\br_{3}-\br_2}};\tau_{0}^{\br_{3}-\br_{2}}<\infty\right]=Z^{(q)}(\br_{3}-\br_{2})-q\frac{(W^{(q)}(\br_{3}-\br_{2}))^2}{W^{(q)\prime}(\br_{3}-\br_{2})}.
			\end{align}
Therefore, \eqref{bar2}--\eqref{aux_4} give
		$V_{\br_{3}}(\br_{3})=\phi(\br_{3};\{\br_{i}\}_{i=1}^{3})$. Finally, we obtain the result by induction and similar arguments as above.
\end{proof}

\begin{proof}[Proof of Proposition \ref{n11}](i) Since $F'(\br^{*}_{1})=0$,  by \eqref{H1}, \eqref{der1} and \eqref{Wp1}, we get that 
\begin{align*}
\dfrac{\sigma^{2} }{2}g'(\br_1^*)+\mu g(\br_1^*)&=\dfrac{\sigma^{2} }{2}g(\br_1^*)\dfrac{W^{(q)\prime\prime}(\br_1^*)}{W^{(q)\prime}(\br_1^*)}+\mu g(\br_1^*)\\
&=\dfrac{g(\br_1^*)}{W^{(q)\prime}(\br_1^*)}\left(\dfrac{\sigma^{2} }{2}W^{(q)\prime\prime}(\br_1^*)-\mu W^{(q)\prime}(\br_1^*)\right)\\
&=qW^{(q)}(\br_1^*)\dfrac{g(\br_1^*)}{W^{(q)\prime}(\br_1^*)}=qH(\br_1^*;\br_1^*).
\end{align*}	
Thus, by \eqref{eqH1}, with $k=0$, we obtain \eqref{HJB11}.

(ii) We first show that
\begin{align}\label{j3}
		H(v;\br^{*}_{1})>g(v)\dfrac{W^{(q)}(v)}{W^{(q)\prime}(v)}=:\Xi(v)\quad\text{for}\ v\in(\br^{*}_{1},\br^{(1)}_{1}).
\end{align}
This follows since
\begin{equation*}
	H(\br^{*}_{1};\br^{*}_{1})=g(\br^{*}_{1})\dfrac{W^{(q)}(\br^{*}_{1})}{W^{(q)\prime}(\br^{*}_{1})}=\Xi(\br^{*}_{1}),
\end{equation*} 
and  \eqref{W5} implies that for $\ v\in(\br^{*}_{1},\br^{(1)}_{1})$
\begin{align*}
\dfrac{\der}{\der v}\Xi(v)&=g(v)+\dfrac{W^{(q)}(v)}{W^{(q)\prime}(v)}\left(g'(v)-\dfrac{g(v)W^{(q)\prime\prime}(v)}{W^{(q)\prime}(v)}\right)\\
&<g(v)=\dfrac{\der}{\der v}H(v;\br^*_{1}).
\end{align*}
Therefore, using again  \eqref{W5}, \eqref{Wp1} and \eqref{j3} we get that \eqref{Hj21} holds on $(\br^{*}_{1},\br^{(1)}_{1})$.
\end{proof}

\begin{proof}[Proof of Lemma \ref{propm1}]
Using \eqref{dnp}, \eqref{par11} is equivalent to $\overline{F}(\Cr_1,z,\br^*_1)>0$ and using \eqref{Wp1}, this is equivalent to
\begin{equation*}
\Gamma(z)\eqdef g'(z)-qH(\Cr_{1};\br^{*}_{1})W^{(q)\prime}(z-\Cr_{1})-\dfrac{2}{\sigma^2}F(\Cr_1,z;\br^*_1)\bigg(qW^{(q)}(z-\Cr_{1})-\mu W^{(q)\prime}(z-\Cr_{1})\bigg)>0.
\end{equation*}
Since $\Gamma(\Cr_1)=0$ due to \eqref{W0}, \eqref{eq9p}, \eqref{eqH1} and \eqref{ineq1.0.0}, it is enough to verify that $\Gamma'(\Cr_1)>0$ in order to prove \eqref{par11} by the $C^{1}$-continuity of this function. Then,
\begin{align*}
\Gamma'(z)=&g''(z)-qH(\Cr_{1};\br^{*}_{1})W^{(q)\prime\prime}(z-\Cr_{1})-\dfrac{2}{\sigma^2}F(\Cr_1,z;\br^*_1)\bigg(qW^{(q)\prime}(z-\Cr_{1})-\mu W^{(q)\prime\prime}(z-\Cr_{1})\bigg)\\
&-\dfrac{2}{\sigma^2}\frac{\overline{F}(\Cr_1,z;\br^*_1)}{W^{(q)\prime}(z-\Cr_{1})}\bigg(qW^{(q)}(z-\Cr_{1})-\mu W^{(q)\prime}(z-\Cr_{1})\bigg)
\end{align*}
and taking $z\downarrow\Cr_1$, by \eqref{eqH1} and \eqref{ineq1.0.0} we obtain that
\begin{align*}
\Gamma'(\Cr_1)&=g''(\Cr_1)+qH(\Cr_{1};\br^{*}_{1})\dfrac{4\mu}{\sigma^4}-\dfrac{2q}{\sigma^2}g(\Cr_1)-\dfrac{4\mu^2}{\sigma^4}g(\Cr_1)\\
&=g''(\Cr_1)+qH(\Cr_{1};\br^{*}_{1})\dfrac{4\mu}{\sigma^4}-\dfrac{2q}{\sigma^2}g(\Cr_1)-\dfrac{2\mu}{\sigma^2}\left(\dfrac{2q}{\sigma^2}H(\Cr_{1};\br^{*}_{1})-g'(\Cr_1)\right)\\
&=g''(\Cr_1)-\dfrac{2q}{\sigma^2}g(\Cr_1)+\dfrac{2\mu}{\sigma^2}g'(\Cr_1).
\end{align*}
Now, from \eqref{ineq1.0.0} we know that there exists ${\varepsilon}_{1}>0$ such that if $z\in(\Cr_{1},\Cr_{1}+\varepsilon_{1})$, then
$$\dfrac{2q}{\sigma^{2}}H(z;\br^{*}_{1})<g'(z)+\dfrac{2\mu}{\sigma^{2}}g(z)=:h(z).$$
We also know that $\dfrac{2q}{\sigma^{2}}H(\Cr_1;\br^{*}_{1})=h(\Cr_1)$. Hence, by taking the derivative at $z=\Cr_1$ we must have that
$$\dfrac{2q}{\sigma^{2}}g(\Cr_{1})<h'(\Cr_1)=g''(\Cr_{1})+\dfrac{2\mu}{\sigma^{2}} g'(\Cr_{1}),$$
therefore $\Gamma'(\Cr_1)>0$.
\end{proof}

\begin{proof}[Proof of Proposition \ref{pmax1}]
	By \eqref{ex2}, it follows that for $v\in(\br^{*}_{1},\br^{(1)}_{1})$, 
	\begin{align*}
\dfrac{g(\br^{*}_{1})}{W^{(q)\prime}(\br^{*}_{1})}>\dfrac{g(v)}{W^{(q)\prime}(v)}\quad\text{and}\quad g(z)<g(v)\dfrac{W^{(q)\prime}(z)}{W^{(q)\prime}(v)}, \quad \text{for}\ z>v.
\end{align*}
From here and \eqref{H1}, we have that
	\begin{align*}
	F(v,z;\br^{*}_{1})&=\dfrac{g(z)-qH(v;\br^*_{1})W^{(q)}(z-v)}{W^{(q)\prime}(z-v)}\\
	&=\dfrac{g(z)-qg(\br^{*}_{1})\dfrac{W^{(q)}(\br^{*}_{1})}{W^{(q)\prime}(\br^{*}_{1})}W^{(q)}(z-v)-q(G(v)-G(\br^{*}_{1}))W^{(q)}(z-v)}{W^{(q)\prime}(z-v)}\\
	&<g(v)\bigg(\dfrac{W^{(q)\prime}(z)-qW^{(q)}(\br^{*}_{1})W^{(q)}(z-v)}{W^{(q)\prime}(v)W^{(q)\prime}(z-v)}\bigg)-q(G(v)-G(\br^{*}_{1}))\dfrac{W^{(q)}(z-v)}{W^{(q)\prime}(z-v)}=:\Lambda(v,z;\br^{*}_{1}).
	\end{align*}
	Observe that if $z\mapsto\Lambda(v,z;\br^{*}_{1})$ is decreasing for $z>v$, it follows that  $F(v,z;\br^{*}_{1})<\dfrac{\sigma^{2}}{2}g(v)$ for $z>v$, because of $\Lambda(v,v;\br^{*}_{1})=F(v,v;\br_{1}^{*})=\dfrac{\sigma^{2}}{2}g(v)$. Let us verify that $\dfrac{\partial}{\partial z}\Lambda(z,v;\br^{*}_{1})<0$. Calculating the first derivative with respect to $z$, we get that
	\begin{align}\label{W3}
	&\dfrac{g(v)}{W^{(q)\prime}(v)(W^{(q)\prime}(z-v))^{2}}\bigg(\Big(W^{(q)\prime\prime}(z)-qW^{(q)}(\br^{*}_{1})W^{(q)\prime}(z-v)\Big)W^{(q)\prime}(z-v)\notag\\
	&\quad-\Big(W^{(q)\prime}(z)-qW^{(q)}(\br^{*}_{1})W^{(q)}(z-v)\Big)W^{(q)\prime\prime}(z-v)\bigg)\notag\\
	&\quad-\dfrac{q(G(v)-G(\br_{1}^{*}))}{(W^{(q)\prime}(z-v))^{2}}\Big((W^{(q)\prime}(z-v))^{2}-W^{(q)}(z-v)W^{(q)\prime\prime}(z-v)\Big)<0\notag\\
	&\Longleftrightarrow\notag\\
	&\dfrac{g(v)}{W^{(q)\prime}(v)}\Big(W^{(q)\prime\prime}(z)W^{(q)\prime}(z-v)-W^{(q)\prime}(z)W^{(q)\prime\prime}(z-v)\Big)\\
	&\quad-q\bigg(G(v)-G(\br_{1}^{*})+\dfrac{g(v)}{W^{(q)\prime}(v)}W^{(q)}(\br^{*}_{1})\bigg)\Big((W^{(q)\prime}(z-v))^{2}-W^{(q)}(z-v)W^{(q)\prime\prime}(z-v)\Big)<0.\notag
	\end{align}
	Using \eqref{W2}--\eqref{W1}, we see that \eqref{W3} is equivalent to 
	\begin{align}\label{W4}
	&\dfrac{g(v)}{W^{(q)\prime}(v)}\dfrac{4q}{\sigma^{4}}\expo^{(\Phi(q)-\xi_{1})(z-v)}W^{(q)}(v)-q\bigg(G(v)-G(\br_{1}^{*})+\dfrac{g(v)}{W^{(q)\prime}(v)}W^{(q)}(\br^{*}_{1})\bigg)\dfrac{4}{\sigma^{4}} \expo^{(\Phi(q)-\xi_{1})(z-v)}<0\notag\\
	&\Longleftrightarrow\notag\\
	&\dfrac{g(v)}{W^{(q)\prime}(v)}(W^{(q)}(v)-W^{(q)}(\br^{*}_{1}))<G(v)-G(\br_{1}^{*}).
	\end{align}
	Since both sides of \eqref{W4} are zero at $v=\br^{*}_{1}$, \eqref{W4} is equivalent to prove that 
	\begin{align}
	&\dfrac{\der}{\der v}\bigg[\dfrac{g(v)}{W^{(q)\prime}(v)}(W^{(q)}(v)-W^{(q)}(\br^{*}_{1}))\Bigg]<\dfrac{\der}{\der v}[G(v)-G(\br_{1}^{*})]\notag\\
	&\Longleftrightarrow\notag\\
	&\dfrac{\Big(W^{(q)}(v)-W^{(q)}(\br^{*}_{1})\Big)}{(W^{(q)\prime}(v))^{2}}\Big(g'(v)W^{(q)\prime}(v)-g(v)W^{(q)\prime\prime}(v)\Big)+\dfrac{g(v)}{W^{(q)\prime}(v)}W^{(q)\prime}(v)<g(v)\notag\\
	&\Longleftrightarrow\notag\\
	&\dfrac{\Big(W^{(q)}(v)-W^{(q)}(\br^{*}_{1})\Big)}{W^{(q)\prime}(v)}\bigg(g'(v)-g(v)\dfrac{W^{(q)\prime\prime}(v)}{W^{(q)\prime}(v)}\bigg)<0\notag.
	\end{align}
  From here and \eqref{W5}, it follows that \eqref{eq21} is true. Letting $v\downarrow\br^{*}_{1}$ it can be easily verified that \eqref{eq21} holds for $\br_1^*$, since $F(\cdot,\cdot;\br^{*}_{1})\in  C(\bar{\mathcal{A}}_{1})$.
\end{proof}

\begin{proof}[Proof of Theorem \ref{prin1}]
Condition \eqref{cond3} immediately implies that $\br^*_2<\Cr_1$, so it only remains to show that $\br^*_2\in\mathcal{D}_1$ to prove the strict inequalities. Suppose that $\br^*_2\notin\mathcal{D}_1$, that is,
\begin{equation}\label{notD1}
F(\br^*_2,z;\br^*_1)<F(\br^*_2,\br^*_2;\br^*_1)=\frac{\sigma^2}{2}g(\br_2^*), \quad \text{for}\ z>\br^*_2.
\end{equation}
Note that \eqref{notD1} is equivalent to
\begin{equation}\label{e14}
	\bar{\Lambda}(\br^{*}_{2},z;\br^{*}_{1})>F(z), \quad \text{for}\ z>\br^*_2,
\end{equation}
where
$$ \bar{\Lambda}(v,z;\br^{*}_{1})\eqdef\dfrac{\sigma^{2}}{2}g(v)\dfrac{W^{(q)\prime}(z-v)}{W^{(q)\prime}(z)}+qH(v;\br^{*}_{1})\dfrac{W^{(q)}(z-v)}{W^{(q)\prime}(z)}, \quad \text{for}\ z\geq v.$$
Also note that for any $\br^*_1\leq v$, $F(v)=\bar{\Lambda}(v,v+;\br^{*}_{1})$.  Additionally, differentiating $\bar{\Lambda}(v,z;\br^{*}_{1})$ with respect to $z$, letting $z\downarrow  v$, and using \eqref{W0}, \eqref{eq1}, \eqref{der1} and \eqref{eqH1}, it follows that
\begin{align*}
	\partial_z\bar{\Lambda}(v,v+;\br^{*}_{1})&=\frac{1}{W^{(q)\prime}(v)}\left(\dfrac{2q}{\sigma^{2}}H(v;\br^{*}_{1})-\dfrac{2\mu}{\sigma^{2}}g(v)-g(v)\frac{W^{(q)\prime\prime}(v)}{W^{(q)\prime}(v)}\right)\\
	&=F'(v)-\frac{1}{W^{(q)\prime}(v)}\dfrac{2}{\sigma^{2}}(\mathcal{L}-q)H(v;\br^*_1).
\end{align*}
Now, using \eqref{cond3}, it yields $\partial_z\bar{\Lambda}(v,v+;\br^{*}_{1})>F'(v)$ for $v\in(\br^{*}_{2},\br^{*}_{2}+\epsilon_{1})$. From here, the smoothness of $\bar{\Lambda}$ and \eqref{e14}, it implies that there exists $\bar{\epsilon}_{1}<\epsilon_{1}$ such that $\bar{\Lambda}(v,z;\br^{*}_{1})>F(z)$, for $v\in(\br^{*}_{2},\br^{*}_{2}+\bar{\epsilon}_{1})$ and  $z>v$, which is equivalent to $F(v,z;\br^{*}_{1})>F(v,v;\br^{*}_{1})$, for $v\in(\br^{*}_{2},\br^{*}_{2}+\bar{\epsilon}_{1})$ and  $z>v$.  Thus, $\br^*_2$ cannot be the infimum of $\mathcal{D}_1$, which is a contradiction. Therefore, $\br^*_2\in\mathcal{D}_1$. Additionally, using again \eqref{cond3} and arguing  similarly as in the proof of Lemma \ref{propm1}, it can be proven that $z\mapsto F(\br^{*}_{2},z;\br^{*}_{1})$ is non-increasing locally at $\br^{*}_{2}$, and concluding that  \eqref{b2} follows due to the continuity of $F(\cdot,\cdot;\br_1^*)$.
\end{proof}

\begin{proof}[Proof of Proposition \ref{n11gen}]
We first show that for $v\in(\br^{*}_{2k-1},\br^{(1)}_{2k-1})$
\begin{multline}\label{j3n}
		H(v;\bu^{*}_{2k-1})\\
		>H(\br^*_{2k-2};\bu^{*}_{2k-3})Z^{(q)}(v-\br^*_{2k-2})+W^{(q)}(v-\br^*_{2k-2})F(\br^{*}_{2k-2},v;\bu^{*}_{2k-3})=:\Xi(v;\bu^{*}_{2k-2}).
\end{multline}
This follows since
\begin{equation*}
	H(\br^{*}_{2k-1};\bu^{*}_{2k-1})=\phi(\br^{*}_{2k-1};\bu^{*}_{2k-1})=\Xi(\br^{*}_{2k-1};\bu^{*}_{2k-2}),
\end{equation*} 
and, \eqref{dnp} and \eqref{W5n} implies that for $\ v\in(\br^{*}_{2k-1},\br^{(1)}_{2k-1})$
\begin{align*}
\dfrac{\der}{\der v}\Xi(v;\bu^{*}_{2k-2})&=qH(\br^*_{2k-2};\bu^{*}_{2k-3})W^{(q)}(v-\br^*_{2k-2})+W^{(q)\prime}(v-\br^*_{2k-2})F(\br^{*}_{2k-2},v;\bu^{*}_{2k-3})\\
&\quad+W^{(q)}(v-\br^*_{2k-2})\frac{\overline{F}(\br^{*}_{2k-2},v;\bu^{*}_{2k-3})}{W^{(q)\prime}(v-\br^*_{2k-2})}\\
&<qH(\br^*_{2k-2};\bu^{*}_{2k-3})W^{(q)}(v-\br^*_{2k-2})+W^{(q)\prime}(v-\br^*_{2k-2})F(\br^{*}_{2k-2},v;\bu^{*}_{2k-3})\\
&=g(v)=\dfrac{\der}{\der v}H(v;\bu^{*}_{2k-1}).
\end{align*}
Therefore, using again \eqref{W5n}, \eqref{Wp1}, \eqref{Zp1} and \eqref{j3n} we get that for $\ v\in(\br^{*}_{1},\br^{(1)}_{2k-1})$
\begin{align*}
\dfrac{\sigma^{2} }{2}g'(v)+\mu g(v)&<\dfrac{\sigma^{2} }{2}\left(g'(v)-\overline{F}(\br^{*}_{2k-2},v;\bu^{*}_{2k-3})\right)\\
&\quad+\mu\left(qH(\br^*_{2k-2};\bu^{*}_{2k-3})W^{(q)}(v-\br^*_{2k-2})+W^{(q)\prime}(v-\br^*_{2k-2})F(\br^{*}_{2k-2},v;\bu^{*}_{2k-3})\right)\\
&=qH(\br^*_{2k-2};\bu^{*}_{2k-3})\left(\mu W^{(q)}(v-\br^*_{2k-2})+\dfrac{\sigma^{2} }{2}W^{(q)\prime}(v-\br^*_{2k-2})\right)\\
&\quad+F(\br^{*}_{2k-2},v;\bu^{*}_{2k-3})\left(\mu W^{(q)\prime}(v-\br^*_{2k-2})+\dfrac{\sigma^{2} }{2}W^{(q)\prime\prime}(v-\br^*_{2k-2})\right)\\
&=q\left(H(\br^*_{2k-2};\bu^{*}_{2k-3})Z^{(q)}(v-\br^*_{2k-2})+W^{(q)}(v-\br^*_{2k-2})F(\br^{*}_{2k-2},v;\bu^{*}_{2k-3})\right)\\
&=q\Xi(v;\bu^{*}_{2k-2})<qH(v;\bu^{*}_{2k-1}).
\end{align*}	
Equality at $v=\br^{*}_{2k-1}$ holds by equality in \eqref{W5n} at $v=\br^{*}_{2k-1}$.
\end{proof}
\begin{proof}[Proof of Proposition \ref{pmaxn}]
	By \eqref{ex2n}, it follows that for $v\in(\br_{2k-1}^{*},\br_{2k-1}^{(1)})$,
	\begin{align*}
	&\frac{g(z)-qH(\br^{*}_{2k-2};\bu^{*}_{2k-3})W^{(q)}(z-\br^{*}_{2k-2})}{W^{(q)\prime}(z-\br^{*}_{2k-2})}<F(\br^{*}_{2k-2},v;\bu^{*}_{2k-3}), \quad\text{ for}\ z>v,\\
	&\Longleftrightarrow\notag\\
	&g(z)<F(\br^{*}_{2k-2},v;\bu^{*}_{2k-3})W^{(q)\prime}(z-\br^{*}_{2k-2})+qH(\br^{*}_{2k-2};\bu^{*}_{2k-3})W^{(q)}(z-\br^{*}_{2k-2}), \quad\text{ for}\ z>v,
	\end{align*}
	and $-F(\br^{*}_{2k-2},\br^{*}_{2k-1};\bu^{*}_{2k-3})<-F(\br^{*}_{2k-2},v;\bu^{*}_{2k-3})$. From here, \eqref{phin}, \eqref{Hn} and \eqref{Fn1}, we have that
		\begin{align*}
		F(v,z;\bu^{*}_{2k-1})&=\dfrac{1}{W^{(q)\prime}(z-v)}[g(z)-qH(v;\br^*_{2k-1})W^{(q)}(z-v)]\\
		&=\dfrac{1}{W^{(q)\prime}(z-v)}\bigg[g(z)-qW^{(q)}(z-v)\bigg[H(\br^{*}_{2k-2};\bu^{*}_{2k-3})Z^{(q)}(\br^{*}_{2k-1}-\br^{*}_{2k-2})\\
		&\quad+W^{(q)}(\br^{*}_{2k-1}-\br^{*}_{2k-2})F(\br^{*}_{2k-2},\br^{*}_{2k-1};\bu^{*}_{2k-3})+G(v)-G(\br^{*}_{2k-1})\bigg]\bigg]\\
		&=\dfrac{1}{W^{(q)\prime}(z-v)}\bigg[g(z)-qW^{(q)}(z-v)W^{(q)}(\br^{*}_{2k-1}-\br^{*}_{2k-2})F(\br^{*}_{2k-2},\br^{*}_{2k-1};\bu^{*}_{2k-3})\\
		&\quad-qW^{(q)}(z-v)\bigg[H(\br^{*}_{2k-2};\bu^{*}_{2k-3})Z^{(q)}(\br^{*}_{2k-1}-\br^{*}_{2k-2})+G(v)-G(\br^{*}_{2k-1})\bigg]\bigg]\\
		&<\dfrac{1}{W^{(q)\prime}(z-v)}\bigg[F(\br^{*}_{2k-2},v;\bu^{*}_{2k-3})W^{(q)\prime}(z-\br^{*}_{2k-2})+qH(\br^{*}_{2k-2};\bu^{*}_{2k-3})W^{(q)}(z-\br^{*}_{2k-2})\\
		&\quad-qW^{(q)}(z-v)W^{(q)}(\br^{*}_{2k-1}-\br^{*}_{2k-2})F(\br^{*}_{2k-2},v;\bu^{*}_{2k-3})\\
		&\quad-qW^{(q)}(z-v)\bigg[H(\br^{*}_{2k-2};\bu^{*}_{2k-3})Z^{(q)}(\br^{*}_{2k-1}-\br^{*}_{2k-2})+G(v)-G(\br^{*}_{2k-1})\bigg]\bigg]\\
		&=\dfrac{1}{W^{(q)\prime}(z-v)}\bigg[F(\br^{*}_{2k-2},v;\bu^{*}_{2k-3})\Big(W^{(q)\prime}(z-\br^{*}_{2k-2})-qW^{(q)}(z-v)W^{(q)}(\br^{*}_{2k-1}-\br^{*}_{2k-2})\Big)\\
		&\quad+qH(\br^{*}_{2k-2};\bu^{*}_{2k-3})\Big(W^{(q)}(z-\br^{*}_{2k-2})-W^{(q)}(z-v)Z^{(q)}(\br^{*}_{2k-1}-\br^{*}_{2k-2})\Big)\\
		&\quad-qW^{(q)}(z-v)\Big(G(v)-G(\br^{*}_{2k-1})\Big)\bigg]=:\Lambda(v,z;\bu^{*}_{2k-1}).
		\end{align*}
		Observe that if $\Lambda(v,z;\bu^{*}_{2k-1})$ is decreasing with respect to $z$, it follows that  $F(v,z;\bu^{*}_{2k-1})<\dfrac{\sigma^{2}}{2}g(v)$ for $z>v$, because of $\Lambda(v,v;\bu^{*}_{2k-1})=\dfrac{\sigma^{2}}{2}g(v)$. Let us verify that $\dfrac{\partial}{\partial z}\Lambda(z,v;\bu^{*}_{2k-1})<0$. Calculating the first derivative with respect to $z$, we get that
		\begin{align}
		&\dfrac{F(\br^{*}_{2k-2},v;\bu^{*}_{2k-3})}{[W^{(q)\prime}(z-v)]^{2}}\Big(\Big(W^{(q)\prime\prime}(z-\br^{*}_{2k-2})-qW^{(q)\prime}(z-v)W^{(q)}(\br^{*}_{2k-1}-\br^{*}_{2k-2})\Big)W^{(q)\prime}(z-v)\notag\\
		&\quad-W^{(q)\prime\prime}(z-v)\Big(W^{(q)\prime}(z-\br^{*}_{2k-2})-qW^{(q)}(z-v)W^{(q)}(\br^{*}_{2k-1}-\br^{*}_{2k-2})\Big)\Big)\notag\\
		&\quad+q\dfrac{H(\br^{*}_{2k-2};\bu^{*}_{2k-3})}{[W^{(q)\prime}(z-v)]^{2}}\Big(\Big(W^{(q)\prime}(z-\br^{*}_{2k-2})-W^{(q)\prime}(z-v)Z^{(q)}(\br^{*}_{2k-1}-\br^{*}_{2k-2})\Big)W^{(q)\prime}(z-v)\notag\\
		&\quad-W^{(q)\prime\prime}(z-v)\Big(W^{(q)}(z-\br^{*}_{2k-2})-W^{(q)}(z-v)Z^{(q)}(\br^{*}_{2k-1}-\br^{*}_{2k-2})\Big)\Big)\notag\\
		&\quad-q\dfrac{\Big(G(v)-G(\br^{*}_{2k-1})\Big)}{[W^{(q)\prime}(z-v)]^{2}}\Big([W^{(q)\prime}(z-v)]^{2}-W^{(q)\prime\prime}(z-v)W^{(q)}(z-v)\Big)<0\notag\\
		&\Longleftrightarrow\notag\\
		&F(\br^{*}_{2k-2},v;\bu^{*}_{2k-3})\Big( W^{(q)\prime\prime}(z-\br^{*}_{2k-2})W^{(q)\prime}(z-v)-W^{q\prime\prime}(z-v)W^{(q)\prime}(z-\br^{*}_{2k-2})\notag\\
		&\quad-qW^{(q)}(\br^{*}_{2k-1}-\br^{*}_{2k-2})\Big([W^{(q)\prime}(z-v)]^{2}-W^{q\prime\prime}(z-v)W^{(q)}(z-v)\Big)\Big)\notag\\
		&\quad+qH(\br^{*}_{2k-2};\bu^{*}_{2k-3})\Big(W^{(q)\prime}(z-\br^{*}_{2k-2})W^{(q)\prime}(z-v)-W^{(q)\prime\prime}(z-v)W^{(q)}(z-\br^{*}_{2k-2})\notag\\
		&\quad-Z^{(q)}(\br^{*}_{2k-1}-\br^{*}_{2k-2})\Big([W^{(q)\prime}(z-v)]^{2}-W^{(q)\prime\prime}(z-v)W^{(q)}(z-v)\Big)\Big)\notag\\
		&\quad-q\Big(G(v)-G(\br^{*}_{2k-1})\Big)\Big([W^{(q)\prime}(z-v)]^{2}-W^{(q)\prime\prime}(z-v)W^{(q)}(z-v)\Big)<0\notag\\
		&\Longleftrightarrow\notag\\
		&F(\br^{*}_{2k-2},v;\bu^{*}_{2k-3})\Big( W^{(q)\prime\prime}(z-\br^{*}_{2k-2})W^{(q)\prime}(z-v)-W^{q\prime\prime}(z-v)W^{(q)\prime}(z-\br^{*}_{2k-2})\Big)\label{W1.1}\\
		&\quad+qH(\br^{*}_{2k-2};\bu^{*}_{2k-3})\Big(W^{(q)\prime}(z-\br^{*}_{2k-2})W^{(q)\prime}(z-v)-W^{(q)\prime\prime}(z-v)W^{(q)}(z-\br^{*}_{2k-2})\Big)\notag\\
		&\quad-q\Big(F(\br^{*}_{2k-2},v;\bu^{*}_{2k-3})W^{(q)}(\br^{*}_{2k-1}-\br^{*}_{2k-2})+H(\br^{*}_{2k-2};\bu^{*}_{2k-3})Z^{(q)}(\br^{*}_{2k-1}-\br^{*}_{2k-2})\notag\\
		&\quad+G(v)-G(\br^{*}_{2k-1})\Big)\Big([W^{(q)\prime}(z-v)]^{2}-W^{(q)\prime\prime}(z-v)W^{(q)}(z-v)\Big)<0.\notag
		\end{align}
		Using \eqref{W2}--\eqref{W7}, we see that \eqref{W1.1} is equivalent to
		\begin{align}
		&F(\br^{*}_{2k-2},v;\bu^{*}_{2k-3})\dfrac{4q}{\sigma^{4}}\expo^{(\Phi(q)-\xi_{1})(z-v)}W^{(q)}(v-\br^{*}_{2k-2})+qH(\br^{*}_{2k-2};\bu^{*}_{2k-3})\frac{4\expo^{(\Phi(q)-\xi_{1})(z-v)}}{\sigma^{4}}Z^{(q)}(v-\br^{*}_{2k-2})\notag\\
		&\quad-q\Big(F(\br^{*}_{2k-2},v;\bu^{*}_{2k-3})W^{(q)}(\br^{*}_{2k-1}-\br^{*}_{2k-2})+H(\br^{*}_{2k-2};\bu^{*}_{2k-3})Z^{(q)}(\br^{*}_{2k-1}-\br^{*}_{2k-2})\notag\\
		&\quad+G(v)-G(\br^{*}_{2k-1})\Big)\dfrac{4}{\sigma^{4}}\expo^{(\Phi(q)-\xi_{1})(z-v)}<0\notag\\
		&\Longleftrightarrow\notag\\
		&F(\br^{*}_{2k-2},v;\bu^{*}_{2k-3})W^{(q)}(v-\br^{*}_{2k-2})+H(\br^{*}_{2k-2};\bu^{*}_{2k-3})Z^{(q)}(v-\br^{*}_{2k-2})\notag\\
		&\quad-\Big(F(\br^{*}_{2k-2},v;\bu^{*}_{2k-3})W^{(q)}(\br^{*}_{2k-1}-\br^{*}_{2k-2})+H(\br^{*}_{2k-2};\bu^{*}_{2k-3})Z^{(q)}(\br^{*}_{2k-1}-\br^{*}_{2k-2})+G(v)-G(\br^{*}_{2k-1})\Big)<0\notag\\
		&\Longleftrightarrow\notag\\
		&F(\br^{*}_{2k-2},v;\bu^{*}_{2k-3})[W^{(q)}(v-\br^{*}_{2k-2})-W^{(q)}(\br^{*}_{2k-1}-\br^{*}_{2k-2})]+H(\br^{*}_{2k-2};\bu^{*}_{2k-3})Z^{(q)}(v-\br^{*}_{2k-2})\label{W4.1}\\
		&\quad<H(\br^{*}_{2k-2};\bu^{*}_{2k-3})Z^{(q)}(\br^{*}_{2k-1}-\br^{*}_{2k-2})+G(v)-G(\br^{*}_{2k-1}).\notag
		\end{align}
Since both sides of \eqref{W4.1} are equal at $v=\br^{*}_{2k-1}$, \eqref{W4.1} is equivalent to prove that 
		\begin{align}
		&\dfrac{\der}{\der v}[F(\br^{*}_{2k-2},v;\bu^{*}_{2k-3})[W^{(q)}(v-\br^{*}_{2k-2})-W^{(q)}(\br^{*}_{2k-1}-\br^{*}_{2k-2})]+H(\br^{*}_{2k-2};\bu^{*}_{2k-3})Z^{(q)}(v-\br^{*}_{2k-2})]\notag\\
		&\quad<\dfrac{\der}{\der v}[H(\br^{*}_{2k-2};\bu^{*}_{2k-3})Z^{(q)}(\br^{*}_{2k-1}-\br^{*}_{2k-2})+G(v)-G(\br^{*}_{2k-1})]\notag\\
		&\Longleftrightarrow\notag\\
		&[W^{(q)}(v-\br^{*}_{2k-2})-W^{(q)}(\br^{*}_{2k-1}-\br^{*}_{2k-2})]\dfrac{\der}{\der v}F(\br^{*}_{2k-2},v;\bu^{*}_{2k-3})\notag\\
		&+F(\br^{*}_{2k-2},v;\bu^{*}_{2k-3})W^{(q)\prime}(v-\br^{*}_{2k-2})+qH(\br^{*}_{2k-2};\bu^{*}_{2k-3})W^{(q)}(v-\br^{*}_{2k-2})<g(v)\notag\\
		&\Longleftrightarrow\notag\\
		&\dfrac{[W^{(q)}(v-\br^{*}_{2k-2})-W^{(q)}(\br^{*}_{2k-1}-\br^{*}_{2k-2})]}{W^{(q)\prime}(v-\br^{*}_{2k-2})}\dfrac{\der}{\der v}F(\br^{*}_{2k-2},v;\bu^{*}_{2k-3})\notag\\
		&+F(\br^{*}_{2k-2},v;\bu^{*}_{2k-3})<\dfrac{g(v)-qH(\br^{*}_{2k-2};\bu^{*}_{2k-3})W^{(q)}(v-\br^{*}_{2k-2})}{W^{(q)\prime}(v-\br^{*}_{2k-2})}\notag\\
		&\Longleftrightarrow\notag\\
		&\dfrac{[W^{(q)}(v-\br^{*}_{2k-2})-W^{(q)}(\br^{*}_{2k-1}-\br^{*}_{2k-2})]}{W^{(q)\prime}(v-\br^{*}_{2k-2})}\dfrac{\der}{\der v}F(\br^{*}_{2k-2},v;\bu^{*}_{2k-3})<0.\notag
		\end{align}
		From here and using \eqref{W5n}, it follows that \eqref{eq21n} is true. Hence $(\br^{*}_{2k-1},\br^{(1)}_{2k-1})\subset\mathcal{D}_{2k-1}^{c}\neq\emptyset$. Letting $v\downarrow\br^{*}_{2k-1}$ in \eqref{eq21}, it can be verified easily that \eqref{eq21n} also holds for $v=\br^{*}_{2k-1}$ due to $F(\cdot,\cdot;\br^{*}_{2k-3})\in C(\bar{\mathcal{A}}_{2k-3})$.
\end{proof}
\begin{proof}[Proof of Proposition \ref{decFF}]
Given $x>\br^*_{2k-1}$, let $\bu^x=\{\bu^*_{2k-3}\}\cup\{\br^*_{2k-2},x\}$. Then, using the definition of $\br^*_{2k-1}$ and the fact that $F(\br^*_{2k-2},z;\bu^*_{2k-3})$ is non-decreasing in $z$, we can show that $V_{\bu^x}(x)=V_{\bu^*_{2k-1}}(x)$, $V'_{\bu^x}(x)=V'_{\bu^*_{2k-1}}(x)$, and $V''_{\bu^x}(x)=V''_{\bu^*_{2k-1}}(x)$. Therefore, we obtain that
$$(\mathcal{L}-q)(V_{\bu^*_{2k-1}}-V_{\bu^x})(x)\leq0.$$
Hence, as in Theorem \ref{Ta1} we get the result.
\end{proof}
\end{document}